\documentclass{amsart}
\usepackage[utf8]{inputenc}
\usepackage{amssymb}
\usepackage[hyperref,nonotes]{degt}

\input NL.def

\def\4{\color{red}}
\def\3{\color{cyan}}
\let\2\relax
\let\3\relax
\let\4\relax
\def\MM{\mathcal M}
\def\bL{\mathbf L}
\def\CU{\mathcal U}

\allowdisplaybreaks

\let\geq\ge
\let\leq\le

\def\sizemax{405}
\def\sizesub{357}

\def\PP{\mathbb P}
\def\cS{\mathcal S}
\def\fT{\frak T}

\let\goal=\gamma
\def\name{351}

\def\pBnd{\tilde\Bnd}
\def\bset{\frak{b}}
\def\Bset{\frak{B}}

\def\prefix{3}
\mathchardef\mincount=350
\def\minusform#1#2{\frac{\bar{#1}}{#2}}

\def\mpat{\tilde\pat}
\def\mval{\frak t}

\def\loccit{\latin{loc\PERIOD~cit}}

\def\ltitle{\subsection{The lattice \lname{\thelattice}}}
\def\ltext{%
 There are \lcount{\thelattice} pairs $\bN\ni\hbar$, and
 $\bnd(\fF)\le\lmax{\thelattice}$%
 \iffullversion.\else\ (see \autoref{tab.\thelattice}).\fi
}
\def\latticetext{\iffullversion\ltitle\ltext\else
 \ifnum\lmax{\thelattice}>\mincount\ltitle\ltext\fi\fi
}
\def\atbeginorbits{\iffullversion\else\table[tp]
 \caption{The lattice \lname{\thelattice}\noaux{ (\cf. \autoref{conv.tables})}}\label{tab.\thelattice}%
 \hrule height0pt\relax\vskip-\abovedisplayskip\fi}
\def\atendorbits{\iffullversion\else\vskip-\belowdisplayskip\endtable\fi}

\def\config#1#2{\hyperlink{#1-#2}{$\##2$}}

\def\specvec#1{\bar{#1}}

\def\vv#1{\specvec1_{#1}}

\def\cw#1{[#1]}

\def\units{^\times}
\def\sdif{\mathbin\vartriangle}
\def\IS{\Cal I}
\def\CL{\Cal L}
\def\CA{\Cal A}
\def\CE{\Cal E}

\def\minitab#1{\vcenter\bgroup\rm
 \def\-##1{\setbox0\hbox{$00$}\hbox to\wd0{\hss$##1$\hss}}%
 \let\\\cr
 \halign\bgroup\strut\hss$##$&&#1\hss$##$\hss\cr}
\def\endminitab{\crcr\egroup\egroup}

\title{Planes in cubic fourfolds}

\author{Alex Degtyarev}

\address{%
Bilkent University\\
Department of Mathematics\\
06800 Ankara, Turkey}

\thanks{%
The first author was partially supported by the
T\"{U}B\DOTaccent{I}TAK grant 118F413.
The second author
was supported in part by
the ANR grant ANR-18-CE40-0009 ENUMGEOM.
The third author was supported by the Research
Council of Norway project no. 250104.
}

\email{degt@fen.bilkent.edu.tr}

\author{Ilia Itenberg}

\address{%
Institut de Math\'{e}matiques de Jussieu--Paris Rive Gauche\\
Sorbonne Universit\'{e} \\
4 place Jussieu, 75252 Paris Cedex 5, France}

\email{ilia.itenberg@imj-prg.fr}

\author{John Christian Ottem}

\address{%
Department of Mathematics,
University of Oslo,
Box 1053, Blindern,
0316 Oslo, Norway}

\email{johnco@math.uio.no}

\keywords{%
Cubic fourfold, integral lattice, Niemeier lattice, discriminant form%
}

\subjclass[2010]{%
Primary: 14J35, 14N20;
Secondary: 14N25, 14P25%
}

\begin{document}

\begin{abstract}
We show that the maximal number of planes in a complex smooth cubic fourfold in ${\mathbb P}^5$
is $405$, realized by the Fermat cubic only; the maximal number of real planes
in a real smooth cubic fourfold is $357$, {\2realized by the so-called
Clebsch--Segre cubic.}
Altogether, there are but three (up to projective equivalence) cubics with more than
$350$ planes.
\end{abstract}

\maketitle

\section{Introduction}\label{Intro}

The study of linear spaces in projective hypersurfaces is a classical problem in algebraic geometry.
The 27 lines on a smooth cubic surface {\2in $\Cp3$,}
{\2going} back to A.~Cayley and G. Salmon in the 19th century
\cite{Cayley:cubics}, and {\2at most $64$} lines on quartic surfaces,
going back to B.~Segre~\cite{Segre}, are two of the most famous examples.
In the last decade, there has been a substantial progress in studying and counting lines
and {\2other} low degree rational curves on
{\2polarized} $K3$-surfaces, see \cite{rams.schuett,
degt:lines,
degt:800,
degt:conics,
degt:sextics,
DIS}.
Moreover, thanks to the global Torelli theorem for $K3$-surfaces \cite{Pjatecki-Shapiro.Shafarevich}
and the surjectivity of the period map \cite{Kulikov:periods}, it has also been possible to obtain a complete description of the surfaces with
large configurations of lines {\2or, sometimes, conics}.

While varieties of lines play an important {\3role} in the geometry of hypersurfaces
(especially cubic threefolds \cite{clemens1972intermediate}
and cubic fourfolds \cite{hassett, voisin}), much less is known about linear
spaces of higher dimension.

In this paper, we
{\3study} $2$-planes in smooth cubic fourfolds $X \subset \Cp5$.
Planes in cubic fourfolds have already appeared in many contexts: they
are a key ingredient in Voisin's proof of the global Torelli theorem \cite{voisin}, they define one of the so-called Hassett divisors in the moduli space of cubics \cite{hassett},
and they serve as important examples in connection with the rationality problem for cubic fourfolds.

As in the case of lines on $K3$-surfaces, the plane counting problem for
smooth cubic fourfolds is not strictly enumerative in the sense that a
generic cubic contains no planes at all. Therefore, one looks for
estimates on
the maximal number of planes in a smooth cubic, and, if
possible,
a description of
the cubics realizing this maximum. In fact, the
similarities between $K3$-surfaces and cubic fourfolds are much deeper, as
cubic fourfolds have their own version of the global Torelli theorem
\cite{voisin} and surjectivity of the period map \cite{laza}. This can be
used to {\3 recast} the original geometric problem in purely
lattice-theoretic terms and eventually obtain a complete characterisation of
extremal cubics. Formally, our characterisation is in terms of the periods.
However,
thanks to the injectivity of the period map,
we can
identify the extremal cubics found
by comparing
their configurations of planes with those of known explicit examples.

Our main result is the following theorem.

\theorem[see \autoref{proof.main}]\label{th.main}
{\2Let $X\subset\Cp5$ be a complex smooth cubic fourfold.}
Then, either $X$ has at most $350$ planes, or, up to projective equivalence,
$X$ is
\roster*
\item
the Fermat cubic {\rm (}with $405$ planes{\rm )}, see \autoref{section:Fermat}
and \eqref{eq.Lmax.1}, or
\item
the Clebsch--Segre cubic {\rm (}with $357$ planes{\rm )},
see \autoref{section:HS} 
and  \eqref{eq.Lsub.2},
or
\item
the {\4\name-cubic} \rom(with $351$ planes\rom),
{\4see \autoref{section:351} and \eqref{eq.Lmisc.1}.}
\endroster
\endtheorem

We also consider a similar problem for real planes in real cubics.
Recall that a {\em real algebraic variety} is a complex algebraic variety $X$ equipped with
a {\em real structure}, \ie, anti-holomorphic involution $c \: X \to X$. A subvariety $P \subset X$
is said to be {\em real} if $c(P) = P$. One can show (\cf. \autoref{construction_maps} and its proof)
that any real structure on a cubic $X \subset \Cp5$ in appropriate coordinates in $\Cp5$
is induced from the coordinatewise complex conjugation.
In these coordinates, both $X$ and a real plane $P \subset X$ can be given
by equations with real coefficients.

For the number of real planes, we have the following stronger bound.

\theorem[see \autoref{proof.main_real}]\label{th.main_real}
The number of real planes in a real smooth cubic $Y \subset \Cp5$
is at most $357$, the equality holding if and only if
$Y$ is
projectively equivalent over $\R$ to the Clebsch--Segre cubic {\rm (}see \autoref{section:HS}{\rm )}.
\endtheorem

Note that the Clebsch--Segre cubic \eqref{CScubic}
can be regarded as a four-dimensional analogue of the Clebsch cubic surface,
which also has all of its 27 lines real.


\remark\label{remark.other_values}
\autoref{th.main} provides a sharp upper bound on the total number of planes.
Another interesting question is that on the possible values that can be taken by the plane count.
It appears that, for smooth cubics, the list is much more sparse than those
counting lines on polarized $K3$-surfaces.
The values observed in our computation are
\catcode`\[\active\def[#1,#2]{#1\mathbin{\smskip..\smskip}#2}%
\def\smskip{\mskip-.5\thinmuskip}
\[*
\gathered
[ 0, 225 ], 227, 229, 231, 233, 235, 237, 239, 241, 243, 245, 247, 249,\\
  255, 257, 259, 261, 267, 273, 285, 297, 351, 357, 405,
\endgathered
\]
\catcode`\[=12
but we do not assert that this list is complete.
\endremark

Given the conclusion of Theorem \ref{th.main}, one can also ask about the
maximal number of planes on singular cubic fourfolds. \mnote{JCO: Added nodal
case} There are easy examples of such cubics where the number of planes is
infinite; however, if one restricts to nodal cubic fourfolds, the number is
finite.
{\4Furthermore, we show that, similar to several known results on rational curves
on $K3$-surfaces, the presence of nodes reduces the maximal number of
planes.}

\begin{proposition}[{\4see \autoref{proof.nodal}}]\label{nodal}
The number of planes in a nodal cubic fourfold is at most {\4$302$}.\mnote{\4Dg:
didn't we agree to use ``in''? Changed to $302$.}
\end{proposition}

{\4The best example known to us has $291$ planes, see \autoref{ex.291};
conjecturally, $291$ is the sharp upper bound in the presence of nodes.}

As
mentioned above, our main approach is lattice-theoretic.
More precisely, given an abstract graph, there is a way to decide
whether it is
{\2realized}
by the configuration of planes
{\2in} a smooth cubic fourfold.
However, unlike the case of lines on a polarized $K3$-surface,
we lack geometric intuition (\eg, elliptic pencils)
which would {\3 narrow the search {\2down} to a sufficiently small collection of graphs}.
For this reason, we take a detour and replant a (modified) abstract lattice of algebraic
cycles to a Niemeier lattice (\ie, one of the $24$ positive definite even unimodular lattices of
rank~$24$), the planes mapping to certain vectors of square $4$.
This approach has a number of benefits. First, instead of dealing with abstract graphs
of \latin{a priori} unbounded complexity, we merely consider subsets of several finite sets
known in advance; in particular, this implies the fact (not immediately obvious) that
the number of planes is uniformly bounded. Second, these finite sets
have rich intrinsic structure that can be used in the construction of large
realizable subsets.
Finally, working with known sets, all symmetry groups can be
expressed in terms of permutations, which makes the computation in
\GAP~\cite{GAP4} {\3 very} effective.

The idea of using Niemeier lattices is not new (\cf. \cite{Kondo,Nikulin:degenerations,Nishiyama}).
The novelty of our treatment is in the fact that the original lattice is odd and,
therefore, {\2it} needs to be modified.
Of course, one could have used embeddings to odd unimodular definite lattices of rang $24$,
but their number is overwhelming.

The paper is organized as follows. In \autoref{4-folds}, we fix the terminology
and {\3 recall} a few basic facts related to integral lattices and cubic fourfolds.
{\3 Towards} the end, in \autoref{section:Fermat} and \autoref{section:HS}, we describe
the geometric configurations of planes in the two extremal cubics, \viz. Fermat and Clebsch--Segre.
In \autoref{S.Neimeier}, we replant the lattice of algebraic cycles of a cubic fourfold
to a Niemeier lattice, thus reducing the original geometric problem to an arithmetic
one;
the algorithms used to solve the latter are outlined in \autoref{algorithms}.
In \autoref{S.few}--\autoref{Leech},
the 24 Niemeier lattices are treated one by one;
this is followed by the proofs of
Theorems \ref{th.main} and \ref{th.main_real}
in \autoref{S.proofs}.

\subsection*{Acknowledgments}
We would like to thank K. Hulek, Z. Li and M. Sch\"{u}tt for
{\2a number of} fruitful discussions {\4and to the anonymous referee of this
paper for several valuable remarks.}
Part of this paper was written during the first and second authors research stay
at the {\em Max-Planck-Institut f\"{u}r Mathematik}; we are grateful to this institution
for its hospitality and support.
The third author acknowledges support from the Centre for Advanced Study, Oslo,
and the {\em Motivic Geometry} programme.

\section{Preliminaries}\label{4-folds}

The principal goal of this section is {\2 to fix} the terminology/notation
and to cite,
in a convenient form, a few fundamental results
used in the sequel.

\subsection{Lattices\pdfstr{}{
 {\rm(see~\cite{Nikulin:forms})}}}\label{s.lattice}
A \emph{lattice} is a free abelian group~$L$ of finite rank equipped with a
symmetric bilinear form $b\:L\otimes L\to\Z$.
Since {\2the form} $b$ is assumed fixed (and omitted from the notation), we abbreviate
$x\cdot y:=b(x,y)$ and $x^2:=b(x,x)$.
The
\emph{determinant} $\det L\in\Z$ is the determinant of the Gram matrix of~$b$
in any integral basis; $L$ is called \emph{nondegenerate} (\emph{unimodular})
if $\det L\ne0$ (respectively, $\det L=\pm1$). The \emph{inertia indices}
$\Gs_\pm L$ are those of the quadratic space $L\otimes\R\to\R$,
$x\otimes r\mapsto r^2x^2$.

A sublattice $S \subset L$ is called {\em primitive}
if the quotient $L/S$ is torsion free.
A {\em $d$-polarized lattice} is a lattice equipped with a distinguished element $h$, $h^2 = d$.
Morphisms in the category of (polarized) lattices are called {\em isometries}.

A lattice~$L$ is called \emph{even} if
$x^2=0\bmod2$ for all $x\in L$; otherwise, $L$ is \emph{odd}. A
\emph{characteristic vector} is an element $v\in L$ such that
$x^2=x\cdot v\bmod2$ for all $x\in L$. If $L$ is unimodular, a characteristic
vector exists and is unique $\bmod\,2L$.
If $v\in L$ is characteristic, the orthogonal complement $v^\perp\subset L$
is even.

For lattices of rank $1$ and $2$ we use the abbreviations
\roster*
\item $[a] := \Z x$, $x^2 = a$;
\item $[a, b, c] := \Z x + \Z y$, $x^2 = a$, $x \cdot y = b$, $y^2 = c$.
\endroster
The \emph{hyperbolic plane} $\bU := [0, 1, 0]$
is the unique unimodular even lattice of rank~$2$.

A nondegenerate lattice~$L$ admits a canonical inclusion
\[*
L\into L\dual:=\bigl\{x\in L\otimes\Q\bigm|
 \mbox{$x\cdot y\in\Z$ for all $y\in L$}\bigr\}
\]
to the dual group~$L\dual$.
The finite abelian group $\CL:=\discr L:=L\dual\!/L$ ($q_L$
in~\cite{Nikulin:forms}) is called the \emph{discriminant group} of~$L$.
Clearly, $\ls|\CL|=(-1)^{\Gs_-L}\det L$.
This group is equipped with the nondegenerate symmetric bilinear form
\[*
\CL\otimes\CL\to\Q/\Z,\quad
(x\bmod L)\otimes(y\bmod L)\mapsto(x\cdot y)\bmod\Z,
\]
and, if $L$ is even, its quadratic extension
\[*
\CL\to\Q/2\Z,\quad x\bmod L\mapsto x^2\bmod2\Z.
\]
We denote by $\CL_p:=\discr_pL:=\CL\otimes\Z_p$ the $p$-primary
components of $\discr L$. The $2$-primary component~$\CL_2$ is called \emph{even}
if $x^2\in\Z$ for all order~$2$ elements $x\in\CL_2$; otherwise, $\CL_2$ is
\emph{odd}. The \emph{determinant} $\det\CL_p$ is the determinant of the
``Gram matrix'' of the quadratic form in any minimal set of generators.
(This is equivalent to the alternative definition given
in~\cite{Nikulin:forms}.)
Unless $p=2$ and $\CL_2$ is odd (in which case
the determinant is not defined or used),
we have $\det\CL_p=u_p/\ls|\CL_p|$, where $u_p$ is a well-defined element of
$\Z_p\units/(\Z_p\units)^2$.

The \emph{length} $\ell(\Cal A)$ of a finite abelian group~$\Cal A$ is the
minimal number of generators of~$\Cal A$. We abbreviate $\ell_p(\Cal
A):=\ell(\Cal A\otimes\Z_p)$ for a prime~$p$.

Given a lattice~$L$ and $q\in\Q$, we use the notation $L(q)$ for the same
abelian group with the form $x\otimes y\mapsto q(x\cdot y)$, assuming that it
is still a lattice. We abbreviate $-L:=L(-1)$, and this notation applies to
discriminant forms as well. The notation $nL$, $n \in \NN$, is used for the
orthogonal direct
sum of $n$ copies of~$L$.

A cyclic group $\Z/b$ equipped with a quadratic form $1 \mapsto a \bmod 2\Z$
is denoted by
$$\Bigl[ \dfrac{a}{b} \Bigr]; \quad \text{we assume that $a, b \in \Z$, $\gcd(a, b) = 1$, $ab = 0 \bmod 2$}.$$
Another notation used in the description of the discriminants
is $\CU := \discr \bU(2)$.

A \emph{root} in an even lattice~$L$ is a vector of square~$\pm2$.
A \emph{root system} is a positive definite lattice generated by roots. Any
root system has a unique decomposition into orthogonal direct sum of
irreducible components, which are of types $\bA_n$, $n\ge1$, $\bD_n$,
$n\ge4$, $\bE_6$, $\bE_7$, or~$\bE_8$ (see, \eg, \cite{Bourbaki:Lie}),
according to their \emph{Dynkin diagrams}.

A \emph{Niemeier lattice} is a positive definite unimodular even lattice of
rank~$24$. Up to isomorphism, there are $24$ Niemeier lattices
(see~\cite{Niemeier}): the \emph{Leech lattice}~$\Lambda$, which is root free,
and $23$ lattices \emph{rationally} generated by roots.
In the latter case, the isomorphism class of a lattice $N:=\N(\bR)$ is uniquely
determined by that of its
maximal root system~$\bR$. For more details, see~\cite{Conway.Sloane}.



\subsection{Cubic fourfolds\pdfstr{}{
 {\rm(see~\cite{hassett, laza, voisin})}}}\label{preliminaries}
Let $X$ be a smooth cubic fourfold in $\PP^5$. The middle Hodge numbers of $X$
are as follows:
$$h^{0,4}=h^{4,0}=0, \;\;\; h^{1,3}=h^{3,1}=1, \;\;\; h^{2,2}=21.$$
We
are
concerned with the middle integral cohomology group $H^4(X) := H^4(X; \ZZ)$; via
the Poincar\'{e} duality isomorphism, this group is canonically
identified with $H_4(X; \ZZ)$.
(From now on, unless stated otherwise, all homology and cohomology groups are with coefficients in $\Z$.)
With respect to the intersection form,
$H^4(X)$ is
{\2the unique (up to isomorphism)}
odd unimodular lattice of signature $(21,2)$.
This lattice is canonically $3$-polarized,
and the distinguished class $h_X$, \viz. the square of the hyperplane divisor of $X$,
is characteristic. There is a lattice isomorphism
$$
H^4(X) \simeq \bL:= 21[+1] \oplus 2[-1], \quad h_X \mapsto h: = (1, \ldots, 1, 3, 3).
$$
Alternatively,
$$\bL \simeq 3[+1] \oplus 2\bU \oplus 2\bE_8, \quad h \mapsto (1,1,1) \in 3[+1].$$
In particular, the sublattice
$\bL^0$
of primitive classes
({\ie, the orthogonal complement of $h$ in $\bL$)\mnote{DgIt: modified}
decomposes as
$$
\bL^0\simeq \bA_2 \oplus 2\bU \oplus 2\bE_8.
$$
The choice of a polarized lattice isomorphism $\phi\: H^4(X) \to \bL$
is called a \emph{marking} of the cubic fourfold $X$, and we call $(X,\phi)$ a {\em marked cubic fourfold}.

By definition, the sublattice $M_X := H^{2,2}(X; \C)\cap H^4(X)$ of integral Hodge classes
is primitive in $H^4(X)$.
The Hodge--Riemann relations imply that $M_X$ is positive definite.
By \cite{Voisin:aspects},
the integral Hodge conjecture holds for $X$, so that
$M_X$ is generated (over $\Z$) by {\2the} classes of algebraic surfaces in $X$.

We denote by $T_X := M_X^\perp$ the \emph{transcendental lattice} of $X$.

\subsection{The global Torelli theorem}

The {\em period} of a marked cubic fourfold
$(X,\phi)$ is defined
as the line $\omega_X=\phi(H^{3,1}(X; \C)) \subset \bL^0 \otimes \C$.
Thus, denoting by $\MM$ the moduli space of marked cubic fourfolds $(X,\phi)$, we may define the {\em period map}
\[*
\alignedat3
\mathcal P \:&& \MM\,\,\, & \to\,\,\, {\mathcal D}&&\subset \PP(\bL^0 \otimes \C),\\
&&(X,\phi) &\mapsto[\omega_X]\rlap,
\endalignedat
\]
where ${\mathcal D}$ is the {\em period domain}, \ie,
a distinguished connected component of
$$\bigl\{x \in \PP(\bL^0 \otimes \C) \bigm| x^2=0, \ x \cdot \bar x<0\bigr\},$$
see \cite{hassett:survey}. The component is distinguished by
{\2the so-called}
\emph{positive sign structure}, \ie,
a coherent choice of orientations
of the maximal negative definite subspaces of $\bL \otimes \R$.
More generally, we can consider cubic fourfolds $X$ polarized by a positive definite
polarized
sublattice $h \in K \subset \bL$.
This gives us a moduli space $\MM_K$ of dimension $21-\rank K$ and an associated period domain $\mathcal D_K$,
which is a connected component of
$$\bigl\{x\in \PP(K^\perp \otimes \C) \bigm| x^2=0, \ x \cdot \bar x<0\bigr\}.$$

The following result is a version of the global Torelli theorem for cubic
fourfolds,
{\2 which is due to C. Voisin \cite{voisin}.}
\theorem\label{construction_maps}
Let $X, Y \subset \Cp5$ be two smooth cubic fourfolds, with
{\4the respective classes $h_X$, $h_Y$ and periods $\omega_X$, $\omega_Y$} as above.
Then, an isometry
$$f^*\: H^4(Y) \to H^4(X)$$
is induced by a holomorphic
{\rm (}respectively, anti-holomorphic{\rm )} projective isomorphism $f\: X \to Y$ if and only if
\roster
\item $f^*$ is polarized, \ie, $f^*(h_Y) = h_X$,
\item $f^*(\omega_Y) = \omega_X$ {\rm (}respectively, $f^*(\omega_Y) = \overline\omega_X${\rm )}.
\endroster
If
{\2an} isomorphism $f$ {\2as above} exists, it is unique.
\endtheorem

\begin{proof}
The holomorphic statement is essentially found in \cite{voisin} (see also \cite{Finashin.Kharlamov:chirality}),
and the anti-holomorphic counterpart is immediately obtained by composing with the complex conjugation.
The uniqueness follows from the well-known fact that an automorphism $X\to X$
which acts as the identity on $H^4(X)$
must be the identity (which can be proved using the Lefschetz fixed-point
{\2theorem).}
\end{proof}

\corollary[\cf. {\cite[Lemma 3.8]{DIS}} or {\cite{Finashin.Kharlamov:cubics}}]\label{real_structure}
A smooth cubic $X \subset \Cp5$ admits a real structure identical on $M_X$
if and only if $T_X$ contains a sublattice
$-\bA_1$ or $\bU(2)$.
\done
\endcorollary

A major consequence of \autoref{construction_maps} is the fact that the period map
$\mathcal P \: \MM \to \mathcal D$ is injective; its image was computed by
R. Laza \cite{laza} and E. Looijenga
\cite{looijenga2009period}.

\begin{theorem}[Surjectivity of the period map, see {\cite[Theorem 1.1]{laza}}]\label{polarized_lattices}
A $3$-polarized lattice $M \ni h$ admits an isometry $\Gf$ onto $M_X \ni h_X$
for a smooth cubic $X \subset \Cp5$ if and only if
\begin{enumerate}
\item $M$ is positive definite and $h$ is a characteristic vector in $M$,
\item $M$ admits a primitive embedding into $\bL$ such that $M^\perp$
is even,
\item\label{isotropic-like} there is no element $e \in M$ such that $e^2 = e \cdot h = 1$,
\item there is no element $e \in M$ such that $e^2 = 2$ and $e \cdot h = 0$.
\end{enumerate}
Under these assumptions, $\MM_M\subset \MM$ has codimension $\rank(M)-1$.
\done
\end{theorem}


\subsection{Planes in cubic fourfolds}\label{section.planes}
Our
{\3main result is}
an upper bound on the number of planes in a smooth cubic fourfold. The following lemma appears in Starr's appendix to \cite{Browning}, where it is attributed to O. Debarre.
\begin{lemma}
Any smooth cubic fourfold contains but finitely many planes.
\done
\end{lemma}


Let $X$ be a smooth cubic fourfold, $P \subset X$ a plane, and $p := [P] \in M_X$ its class.
Clearly, $h_X \cdot p = 1$, and using the normal bundle sequence, we find that $p^2=c_2(N_{P|X})=3$.
Furthermore, given two distinct planes $P_1$, $P_2$ with classes $p_1$, $p_2$,
one has the following trichotomy:
\roster*
\item $p_1 \cdot p_2 = 0$ if $P_1$ and $P_2$ are disjoint;
\item $p_1 \cdot p_2 = 1$ if $P_1$ and $P_2$ intersect at a point;
\item $p_1 \cdot p_2 = -1$ if $P_1$ and $P_2$ intersect in a line.
\endroster

This has the following important consequence.
\begin{lemma}\label{uniqueplanes}
Each class $p\in M_X$ is represented by at most one plane.
\done
\end{lemma}

A configuration of planes in a smooth cubic $X \subset \Cp5$ is described by means
of its {\em graph of planes} $\Fn X$: the vertices of $\Fn X$ are planes $P \subset X$,
and two vertices $P_1$, $P_2$ are connected by a solid (respectively, dotted) edge
whenever $P_1$ and $P_2$ intersect at a point (respectively, in a line).
By $\ls|\Fn X|$, we denote the number of vertices of this graph, \ie, the number of planes in $X$.

{\3The next proposition gives
{\2us a}
precise relationship between classes in $M_X$ and planes in $X$.}

\begin{proposition}\label{planeclasses} Given a smooth cubic fourfold $X\subset \PP^5$,
the map $P \mapsto [P]$ establishes
a bijection between
$\Fn X$
and
the set of classes $p \in M_X$ such that $p^2=3$ and $h_X \cdot p=1$.
\end{proposition}

\begin{proof}
The
map $P\mapsto [P]$ is injective by
\autoref{uniqueplanes},
whereas the surjectivity is essentially stated
in the first paragraph of \cite[\S 3]{voisin}. {\3 {Here is a more direct argument}}.

Let $\pi\: \mathcal X\to B$ be a local universal family of marked cubic fourfolds, with  $X$ as the fiber over $b_0\in B$. The marking allows us to identify each lattice $M_{X_b}$, $b\in B$, with a sublattice of $\bL$. Let $p\in M_X\subset \bL$
be a class with $p^2=3$ and $h_X \cdot p=1$, and let $B'\subset B$ denote the
Hodge locus of $p$, parameterizing the fibers $\mathcal X_b$ for which the class
$p$ stays Hodge, \ie, $p\in M_{\mathcal X_b}$.
Note that
\[
\text{the sublattice $K:=\Z h+\Z p\simeq[3,1,3]\subset M_{\mathcal X_b}$ is necessarily primitive}
\label{eq.primitive}
\]
(as its only proper finite
index extension contains a vector as in
\autoref{polarized_lattices}\iref{isotropic-like});
hence,
the closed codimension $1$ subset $B' \subset B$ is irreducible
(see \cite[Theorem 3.2.3]{hassett}). On the other hand,
{\2a simple dimension count shows}
that the cubic fourfolds containing a plane form an irreducible divisor
in the {\3 parameter space of cubics}.

{\2Therefore,}
if $b\in B'$ is a very general point,
{\2then both}
\roster*
\item
the cubic $X'=\mathcal X_b$ contains a plane $P'$ and
\item
the {\2lattice} $M_{X'}$ has rank 2,
hence $M_{X'}=K$ by~\eqref{eq.primitive}.
\endroster
Since also
\roster*
\item
$p\in M_{X'}=K$ is the only class such that $h_{X'}\cdot p=1$ and $p^2=3$,
\endroster
we conclude that $[P']=p$.
Then, by specialization, in $X$ the class $p$ is also
represented by an effective cycle of degree 1 in $\PP^5$,
hence a plane $P\subset X$.\end{proof}

\begin{corollary}\label{planeclasses_real}
Given a real structure $c\: X \to X$ on a smooth cubic fourfold $X$,
a class $p\in M_X$
{\2is represented by}
a real plane $P\subset X$ if and only if
\begin{enumerate}
\item $h_X\cdot p=1$ and $p^2=3$,\label{item_complex}
\item $c^*p=p$.\label{item_real}
\end{enumerate}
\end{corollary}
\begin{proof}
If \iref{item_complex} holds, then $p$ is represented by a unique plane $P$
by \autoref{planeclasses},
and $P$ satisfies $c(P)=P$ by \iref{item_real}.
The converse is immediate.
\end{proof}

\subsection{The Fermat cubic}\label{section:Fermat}
Let $X \subset \PP^5$ be the {\em Fermat cubic}, defined by the equation
\begin{equation}\label{Fcubic}x_0^3+x_1^3+\ldots+x_5^3=0.
\end{equation}
One can easily see that $X$ contains at least $405$ planes.
Indeed, consider one of the $5 \times 3$ splittings of the index set $\{0, \ldots, 5\}$
into three pairs, \eg, $\{0, 1\}, \{2, 3\}, \{4, 5\}$, and pick three cubic roots
$\xi_1$, $\xi_2$, $\xi_3$ of $-1$. Then, each of the planes
$$
x_0 = \xi_1 x_1, \quad x_2= \xi_2 x_3, \quad x_4 = \xi_3 x_5
$$
clearly lies in $X$.
The number of planes obtained in this way is $15 \times 3^3 = 405$.

A direct calculation (e.g., following the argument of Segre \cite[pp\PERIOD~122--123]{segre1944quartic}) shows that $X$ does not contain any other plane.
Alternatively, this statement is an immediate corollary of \autoref{th.main};
moreover, we assert that $X$ is the only (up to projective equivalence) smooth cubic
with $405$ planes.


\subsection{The Clebsch--Segre cubic}\label{section:HS}
The {\em Clebsch--Segre cubic} is the cubic fourfold $Y$
defined by the following equations in $\PP^6$:
\begin{equation}\label{CScubic}
x_0 + x_1 + \ldots +x_6 = x_0^3 + x_1^3 + \ldots + x_6^3 = 0.
\end{equation}Note that the full symmetric group $\SG7$ acts on $Y$ by permuting the coordinates.

According to K.~Hulek and M.~Sch\"{u}tt (private communication), $Y$
contains at least 357 planes, constituting
two $\SG7$-orbits. One orbit consists of the $105$ Fermat-type planes
$$
x_i + x_j = x_k + x_l = x_m + x_n = x_o = 0,
$$
where $i, j, k, l, m, n, o$ is a permutation of $0, 1, 2, 3, 4, 5, 6$.
To describe the other orbit, we fix an {\4ordered}
sequence of three pairwise
disjoint
{\4 couples}
in the index set, \eg, $(1, 2), (3, 4), (5, 6)$.
{\4(Other sequences, as well as all related objects described below, are
obtained by permuting the coordinates \via\ $\SG7$; we choose
to avoid cryptic multi-index constructs in the description.)}\mnote{Dg: to address the
referee's comment}
Consider a vector
$$(0, 1, -1, \Gf, -\Gf, 0, 0) \in \C^7,$$
where $\Gf^2 + \Gf - 1 = 0$,
and denote by $O$ its orbit under the dihedral group $\DG{10} \subset \SG7$
generated  by the simultaneous transposition $1 \leftrightarrow 2$, $3 \leftrightarrow 4$ and
the $5$-cycle
$$
0 \mapsto 1 \mapsto 3 \mapsto 4 \mapsto 2 \mapsto 0.
$$
It is straightforward to check that the image $[O] \subset \PP^6$ of $O$ consists of
five collinear points
and the plane spanned by $[O]$ and $[0 : 0 : 0 : 0 : 0 : 1 : -1]$ lies in $Y$.
This plane is stabilized by $\DG{10} \times \<5 \leftrightarrow 6\>$; hence, its $\SG7$-orbit
is of size 252, resulting in the total of 357 planes in $Y$.

By a direct computation,
or as a consequence of
\autoref{lemma.HS}, we conclude that the cubic $Y$
given by~\eqref{CScubic} contains no other planes.
Observe also that, immediately by the construction, all $357$ planes in $Y$ are real
{\2(as they are spanned by real points).}

\subsection{The \name-cubic}\label{section:351}
The cubic fourfold $Y\subset \PP^5$
{\4with 351 planes} is given by the equation
\begin{equation}\label{equationform}
g(x_0,x_1,x_2)=g(x_3,x_4,x_5),\quad
g(t_0,t_1,t_2):=t_0t_1^2-t_2^3-t_0^2t_2;
\end{equation}
this cubic appears in the
recent paper of K. Koike \cite{koike}.
To
{\4describe}
the planes in~{\4$Y$},
{\4observe} that
{\4$g=0$ defines smooth cubic curves $C_1$ in the
$[x_0:x_1:x_2]$-plane~$P_1$ and $C_2$ in the $[x_3:x_4:x_5]$-plane~$P_2$.}\mnote{Dg:
I tried
to avoid the $V$-notation used in two different meanings, but I do not
insist. Formally, in the old notation,
I think it should've been $C_1\subset P_2$}
For each pair of
flex
{\4tangents $L_1\subset P_1$ and $L_2\subset P_2$,}
the span of $L_1$
and $L_2$
{\4is a space $\PP^3\subset\PP^5$}, and
{\4the intersection} $\PP^3\cap  Y$ is a union of three planes.
Since each cubic curve has 9 inflection points, this gives
{\4us} $3\times 9\times9=243$ planes,
{\4which are all pairwise distinct.}

The reason for the choice of $g$ is that
{\4$C_1\cong C_2$}
have many
automorphisms:
Koike shows that $Y$ contains $108$ additional
planes, one for each automorphism of $C_1$ (see \cite[Lemma 2.2]{koike}).
This
{\4results in} a total of $243+108=\4 351$ planes in $Y$.\mnote{\4Dg: here
and elsewhere: haven't we decided on "planes {\bf in} a fourfold"?}

\remark
Note
that the Fermat cubic fourfold, {\4considered in \autoref{section:Fermat}},
{\4 can also be represented by} \eqref{equationform},
with $g(t_0,t_1,t_2)=t_0^3+t_1^3+t_2^3$, {\4so that the respective cubic
$C_1$ has even more automorphisms.}
\endremark

%
%
%
%

\section{Reduction to the Niemeier lattices}\label{S.Neimeier}

The principal goal of this section is replanting
a lattice $M$ as
in \autoref{polarized_lattices}
to an appropriate Niemeier lattice. Then, in \autoref{s.orbits} and
\autoref{s.idea}, we outline the strategy of our proof of
Theorems~\ref{th.main} and~\ref{th.main_real}.

\subsection{Replanting to a Niemeier lattice}\label{s.Neimeier_replanting}
Consider a positive definite $3$-polarized lattice $M \ni h$.
Assume that $h$ is characteristic in $M$
and $M$ contains at least one $h$-\emph{plane}, \ie,
a vector $l$ such that $l^2 = 3$ and $l \cdot h = 1$.

We denote by $S := S(M)$ the index $3$ extension of the lattice
$h^\perp\oplus \Z\hbar$, $\hbar^2 = 12$, by the vector
$\frac{1}{3}(3l-h+\hbar)$;
this extension does not depend on the choice of an $h$-plane $l$. A
\emph{plane} in $S$ is a vector $l \in S$ such that
$l^2=l\cdot\hbar=4$.

The following statement is immediate.

\proposition\label{prop.S}
The lattice $S$ constructed above has the following properties\rom:
\roster
\item\label{S.even}
$S$ is even and positive definite\rom;
\item\label{S.4}
$\hbar \in 4S\dual$\rom;
\item\label{S.exeptional}
the orthogonal complements $h^\perp \subset M$
and $\hbar^\perp \subset S$ are identical\rom; hence, in particular,
there is a canonical bijection between their sets of roots.
\endroster
Furthermore, the map $x \mapsto \frac{1}{3}(3x - h + \hbar)$
establishes bijections between
\roster[\lastitem]
\item\label{S.planes}
$h$-planes in $M$ and planes in $S$\rom;
\item\label{S.eh=4}
vectors as in \autoref{polarized_lattices}\iref{isotropic-like}
and $e \in S$
such that $e^2 = 2$, $e \cdot \hbar = 4$.
\done
\endroster
\endproposition

\remark\label{rem.roots}
In view of item~\iref{S.4} in \autoref{prop.S}, for any root $e\in S$, we have
$e\cdot\hbar\in\{0,\pm4\}$; hence, either $e\in\hbar^\perp$ or $\pm e$ is as
in item~\iref{S.eh=4}.
We are mainly interested in lattices $S$ not containing either of these two classes of vectors;
it follows that this condition is equivalent to the requirement that $S$ should be root free.
\endremark

\proposition\label{prop.replanting}
If $M$ admits a primitive embedding to $\bL$ with even orthogonal complement,
then $S$ admits an embedding to a Niemeier lattice $N$ such that
the torsion of $N/S$ is a $2$-group.
\endproposition

\proof The proof relies upon Nikulin's theory of discriminant forms (see \cite{Nikulin:forms}).
Let $\rho = \rank M$ and $T = M^\perp \subset \bL$.
Then,
since $h^\perp \subset M$ is the orthogonal complement of $\Z h \oplus T$ in the unimodular lattice $\bL$,
we have
$$\discr h^\perp \simeq
\<\eta\>
\oplus \discr(-T),
$$
where $3\eta = 0$ and $\eta^2 = \tfrac{2}{3} \bmod 2\Z$;
we use the assumption that $M$ contains an $h$-plane, so that $h \notin 3M\dual$
and, hence, the sublattice $\Z h \oplus T$ is primitive in $\bL$.
The passage from $h^\perp$ to $S$ changes the discriminant to
$$\discr S = \bigl\<\tfrac{1}{4}\hbar\bigr\> \oplus \discr(-T), \quad
\bigl(\tfrac{1}{4}\hbar\bigr)^2 = \tfrac{3}{4} \bmod 2\Z,$$
which almost satisfies the hypotheses of
{\2\cite[Theorem 1.12.2]{Nikulin:forms}.}
The only difficulty is in item (4) of
{\2\loccit.:}
if $\ell(\discr_2 T) = \rank T$
and $\discr_2 T$ is even, then $\discr_2 S$ is also even and has a wrong determinant
$$\pm 3 \ls|\discr S| \bmod (\Q_2^\times)^2$$
instead of $\pm \ls|\discr S| \bmod (\Q_2^\times)^2$.
However, since $\ell(\discr_2 S) = 24 - \rho \geq 3$, we may pass to an appropriate iterated index $2$ extension
and either make $\discr S$ odd or reduce its length.
{\4Indeed,\mnote{Dg: for the referee}
the classification of discriminant $2$-forms (see, \eg,
\cite{Nikulin:forms}) implies that, with the exception of $\pm3\bigl[\frac12\bigr]$
and a few forms of length $\ell\le2$, any such form contains an isotropic
element and, hence, can be reduced to a smaller one.}
\endproof

The construction above is invertible: starting from a pair $S \ni \hbar$, where $S$ is a positive definite even lattice
and $\hbar^2 = 12$, and assuming that $\hbar \in 4S\dual$,
one can construct a unique $3$-polarized lattice $M \ni h$ such that $S = S(M)$.
However, the converse of \autoref{prop.replanting} does not hold,
and we state it as an extra restriction.

\proposition\label{Niemeier_back}
Let $S$ be a positive definite even lattice, and let $\hbar \in S \cap 4S\dual$ be a vector of square $12$.
Then, the $3$-polarized lattice $M \ni h$ obtained from $S \ni \hbar$ by the inverse construction
admits a primitive embedding to $\bL$ with even orthogonal complement
if and only if\rom:
\roster
\item\label{cond1} $\rank S \leq 21$\rom; we denote $\delta := 23 - \rank S \geq 2$, and
\endroster
the discriminant $\cS := \discr S$ has the following properties at each prime $p$\rom:
\roster[\lastitem]
\item\label{cond2}
if $p > 2$, then either $\ell(\cS_p) < \delta$ or $\ell(\cS_p) = \delta$ and $\det \cS_p = \ls|\cS| \bmod (\Q_p^\times)^2$\rom;
\item\label{cond3}
if $p = 2$, then $\ell(\cS_2) \leq \delta + 1$ and, in the case of equality $\ell(\cS_2) = \delta + 1$, either $\cS$ is odd or
$\det \cS_2 = \pm 3 \ls|\cS| \bmod (\Q_2^\times)^2$.
\endroster
\endproposition

\proof
We use
{\2\cite[Theorem 1.12.2]{Nikulin:forms}.}
Under the assumption that $\hbar \in 4S\dual$,
there is a splitting
$\cS = \bigl\<\frac{1}{4}\hbar\bigr\> \oplus {\mathcal T}$, and we merely
restate the restrictions
on ${\mathcal T} \cong \discr M$ in terms of $\cS$.
\endproof

\subsection{Admissible and geometric sets}\label{s.admissible_sets}
In view of Propositions \ref{prop.S} and \ref{prop.replanting},
we can replace the lattice $M_X$ of a smooth cubic $X \subset \Cp5$
by the corresponding lattice $S(M_X)$ and construct the latter
directly in an appropriate Niemeier lattice,
as the span of its set of planes.

Thus, we fix a Niemeier lattice $N$ and a vector $\hbar \in N$, $\hbar^2 = 12$,
and consider the set of \emph{planes}
$$
\fF(\hbar) = \bigl\{l \in N \bigm | l^2 = 4,\ l \cdot \hbar = 4\bigr\}.
$$
For any subset $\fL \subset \fF(\hbar)$, we define its \emph{span}
$$
\spn_2 \fL = (\Z_2 \fL + \Z_2 \hbar) \cap N \subset{\2N}
$$
{\2(where the intersection is in $N \otimes \Z_2$).}
The \emph{rank} of $\fL$ is $\rank \fL := \rank \spn_2 \fL$,
and we say that $\fL$ is \emph{generated} by a subset $\fL' \subset \fL$
if $\fL = \fF(\hbar) \cap \spn_2 \fL'$.

By definition, the torsion of $N/\spn_2 \fL$ is a $2$-group
and $\hbar \in 4(\spn_2 \fL)\dual$. A finite index extension $S \supset \spn_2 \fL$ in $N$
is called \emph{mild} if $\hbar \in 4S\dual$ and $S$ is root free, \cf.
\autoref{rem.roots}.

A subset $\fL \subset \fF(\hbar)$ is called \emph{complete}
if $\fL = \fF(\hbar) \cap \spn_2 \fL$; it is called \emph{saturated}
if the identity $\fL = \fF(\hbar) \cap S$ holds for any mild extension $S$ of $\spn_2 \fL$.

\definition\label{def:admissible}
A subset $\fL \subset \fF(\hbar)$ is called \emph{admissible}
if $\spn_2 \fL$
is root free, \cf. \autoref{rem.roots}.
A complete admissible subset $\fL$ is \emph{pseudo-geometric} if
$S = \spn_2 \fL$ satisfies conditions \iref{cond1} and \iref{cond2} in \autoref{Niemeier_back};
it is called \emph{geometric} if $\spn_2 \fL$ admits a mild extension $S$
satisfying all hypotheses of \autoref{Niemeier_back}
and such that $\fL=\fF(\hbar)\cap S$.
\enddefinition

Since the lattice $N$ is positive definite, we have $-1 \leq l_1 \cdot l_2 \leq 3$
for any two distinct planes $l_1$ and $l_2$.
If $\fL \subset \fF(\hbar)$ is admissible, then
\[
\mbox{$l_1 \cdot l_2 \in \{0, 1, 2\}$ for any distinct planes $l_1, l_2 \in \fL$.}
\label{eq.weaker_property}
\]
(Indeed, if $l_1 \cdot l_2 = 3$ or $-1$, then, respectively, $e:=l_1-l_2$ or
$\hbar-l_1-l_2$ is a root.)
Thus, we can regard $\fL$ as a graph, connecting two vertices $l_1 \ne l_2$
by an edge of multiplicity $l_1 \cdot l_2 - 1$ (\cf. the description of the
graph of planes
in \autoref{section.planes}).

Two $12$-polarized Niemeier lattices $N_i\ni\hbar_i$, $i=1,2$,
are called \emph{equivalent} if the corresponding polarized lattices
$\spn_2\fF(\hbar_i)\ni\hbar_i$ are isomorphic
and both torsions $\Tors(N_i/\fF(\hbar_i))$ are $2$-groups.
Clearly,
equivalent polarized lattices share the same collections of admissible and
pseudo-geometric sets of planes. \latin{A priori}, this is not true for
geometric sets.

\subsection{Orbits, counts, and bounds}\label{s.orbits}
Let $\bN\ni\hbar$ be a $12$-polarized Niemeier lattice
as in \autoref{s.admissible_sets}, and let
$\OG(\bN)\supset\RG(\bN)$ be the full orthogonal group of~$\bN$ and its
subgroup generated by reflections, respectively. We denote by
$\OG_\hbar(\bN)\supset\RG_\hbar(\bN)$ the stabilizers of~$\hbar$ in these two
groups. The stabilizers act on $\fF(\hbar)$ and, hence, $\fF(\hbar)$ splits into
$\OG_\hbar(\bN)$-\emph{orbits} $\borb_n$, each orbit splitting into
$\RG_\hbar(\bN)$-orbits $\orb\subset\borb_n$, called \emph{combinatorial
orbits}. The number of combinatorial orbits in an orbit $\borb_n$ is
denoted by~$m(\borb_n)$, and the set of all combinatorial orbits is denoted by
$\Orb:=\Orb(\hbar)$. This set inherits a natural action of the group
\[*
\stab\hbar:=\OG_\hbar(\bN)/\!\RG_\hbar(\bN),
\]
which preserves each orbit~$\borb_n$.
(By an obvious abuse of notation, occasionally $\borb_n$ is treated as
a subset of $\Orb$, whereas subsets of $\Orb$ are treated
as sets of planes.)

For a subset $\Cluster\subset\Orb$, let
\[*
\Bset(\Cluster):=\bigl\{\fL\cap\Cluster\bigm|
 \mbox{$\fL\subset\fF(\hbar)$ is pseudo-geometric}\bigr\},\quad
\bset(\Cluster):=\bigl\{\ls|\fL|\bigm|\fL\in\Bset(\Cluster)\bigr\}.
\]
Since the pseudo-geometric property is obviously inherited by complete subsets,
in the computation we can confine ourselves to the
pseudo-geometric sets $\fL\subset\fF(\hbar)$ generated by $\fL\cap\Cluster$.

Following~\cite{degt:sextics}, define the \emph{count} $\cnt(\orb)$ and
\emph{bound} $\bnd(\orb)$
of a single combinatorial orbit $\orb\in\Orb$ \via
\[*
\cnt(\orb):=\ls|\orb|,\qquad
\bnd(\orb):=\max\bset(\orb).
\]
Usually, the bound $\bnd(\orb)$ is replaced with its rough estimate computed as
explained in \autoref{s.blocks} below.
Clearly, $\cnt$ and $\bnd$ are constant within each orbit~$\borb_n$.
We extend these notions to subsets $\Cluster\subset\Orb$
\emph{by additivity}:
\[*
\cnt(\Cluster):=\sum_{\orb\in\Cluster}\cnt(\orb),
\qquad
\bnd(\Cluster):=\sum_{\orb\in\Cluster}\bnd(\orb).
\]
Thus, we have a na\"{\i}ve \latin{a priori} bound
\[
\ls|\fL|\le\bnd(\Orb)=\sum m(\borb_n)\bnd(\orb),\quad\orb\subset\borb_n.
\label{eq.naive}
\]
Clearly, the true count $\ls|\fL\cap\Cluster|$ is genuinely additive, whereas
the true sharp bound
$\max\bset(\Cluster)\le\bnd(\Cluster)$
is only subadditive; thus, our
proof of \autoref{th.main} will essentially consist in
reducing~\eqref{eq.naive} down to a certain preset goal~$\goal$.
To this end, we will
consider the set\mnote{It: $\fF(\hbar)$}
\[*
\Bnd=\Bnd(\fF(\hbar)):=
 \bigl\{\fL\subset\fF(\hbar)\bigm|\mbox{$\fL$ is geometric}\}/\!\OG_\hbar(\bN)
\]
and,
for a collection of orbits $\Cluster=\borb_1\cup\ldots$ and
integer $d\in\NN$, let
\[*
\Bnd_d(\Cluster):=\bigl\{[\fL]\in\Bnd\bigm|
 \mbox{$\fL$ is generated by $\fL\cap\Cluster$ and
 $\ls|\fL\cap\Cluster|\ge\bnd(\Cluster)-d$}\bigr\}.
\]
We will also consider the oversets
$\pBnd_d(\Cluster)\supset\Bnd_d(\Cluster)$ consisting of
pseudo-geometric (rather than geometric) sets.
The computation of these sets is discussed in \autoref{s.Bset}.

\subsection{Idea of the proof}\label{s.idea}
Fix
a $12$-polarized Niemeier lattice $\bN\ni\hbar$ and a goal\mnote{It: point added}
\[
\ls|\fL|\ge\goal\quad\mbox{(typically, $\goal = 351$)}.
\label{eq.goal}
\]
We need to list all geometric sets $\fL\subset\fF(\hbar)$ satisfying this
inequality. Clearly, such sets may exist only if $\bnd(\Orb)\ge\goal$,
where $\bnd(\Orb)$ is the na\"ive bound given by~\eqref{eq.naive}; otherwise,
the pair $\bN\ni\hbar$ does not need to be considered.

In the few remaining cases, we use brute force and, for each combinatorial
orbit~$\orb$, compute the $\RG_\hbar(\bN)$-orbits on
the set $\Bset(\orb)$.
(Obviously, it suffices to consider one representative of each
orbit $\borb_n$; the rest is obtained by translation.) In particular, we
obtain sharp bounds $\bnd(\orb)$ and sets of values
\[
\bset(\orb)=\{\bnd(\orb)>\bnd'(\orb)>\ldots\}.
\label{eq.values}
\]
This may yield a better bound $\bnd(\Orb)$ given
by~\eqref{eq.naive}; this improved bound is used in the subsequent
computation.
If still $\bnd(\Orb)\ge\goal$,
we choose a collection of pairwise
disjoint unions of orbits $\Cluster_1,\ldots,\Cluster_m\subset\Orb$
and integers $d_1,\ldots,d_m\ge0$ such that
\[*
d_1+\ldots+d_m+m>\bnd(\Orb)-\goal.
\]
Then, clearly, any geometric set $\fL$ satisfying~\eqref{eq.goal} is in the
union
\[*
\CE:=\Bnd_{d_1}(\Cluster_1)\cup\cdots\cup\Bnd_{d_m}(\Cluster_m),
\]
and the same assertion holds for pseudo-geometric sets, with $\Bnd_d$ replaced
with $\pBnd_d$. We try to fix the choices so that the union~$\CE$ consists of
relatively few sufficiently large sets; then,
these exceptional sets are analyzed
one by one using one of the
{\2arguments described below.}

\subsubsection{Maximal sets\pdfstr{}{
 {\rm(see~\cite{degt:sextics})}}}\label{ss.max.set}
If a set $\fL\in\CE$ is saturated and the rank $\rank\fL=21$ is maximal,
\cf. \autoref{Niemeier_back}\iref{cond1},
then $\fL$ has no proper geometric extension; hence, this set can be either
disregarded, if $\ls|\fL|<\goal$, or listed as an exception in the respective
statement.

\subsubsection{Extension by a maximal orbit\pdfstr{}{
 {\rm(see~\cite{degt:sextics})}}}\label{ss.max.ext}
In many cases, a set $\fL\in\Bnd_d(\Cluster)$ has the property that
\[*
\sum\bigl(\bnd(\orb)-\bnd'(\orb)\bigr)
 \ge\bnd(\Orb)-\goal,\qquad
 \orb\in\Orb_\Gd:=\bigl\{\orb\in\Orb\bigm|\ls|\fL\cap\orb|<\bnd(\orb)\bigr\},
\]
see~\eqref{eq.values}.
This implies that any (pseudo-)geometric extension $\fL'\supset\fL$
satisfying~\eqref{eq.goal} must have maximal intersection,
$\ls|\fL'\cap\orb|=\bnd(\orb)$, with at least one orbit $\orb\in\Orb_\Gd$.
Trying these orbits one by one, we obtain larger sets, which are usually maximal,
see \autoref{ss.max.set}.
Note that here we always assume $\fL\cap\Cluster$ fixed, \ie, we accept only those
extensions $\fL'\supset\fL$ that have the property
\[
\mbox{$\fL'\supset\fL$ is pseudo-geometric and $\fL'\cap\Cluster=\fL\cap\Cluster$};
\label{eq.fixed.ext}
\]
indeed, otherwise, we would have started with a larger set
$\fL''\in\Bnd_d(\Cluster)$, \viz. the one generated by $\fL'\cap\Cluster$.
Thus, for the computation,
we merely extend the restricted pattern
\[*
\pat_\fL\:\Cluster\to\NN,\quad\orb\mapsto\ls|\fL\cap\orb|
\]
(see \autoref{s.patterns} below),
by a single extra value $\orb'\mapsto\bnd(\orb')$ for some orbit
$\orb'\in\Orb_\Gd\sminus\Cluster$ and perform one extra step of the algorithm.
We make use of the symmetry of~$\fL$, trying for~$\orb'$ a single
representative of each orbit of the action on $\Orb_\Gd\sminus\Cluster$ of the
$(\stab\hbar)$-stabilizer of $\pat_\fL$.

\subsubsection{Maximal orbit count}\label{ss.max.orbit}
For smaller sets $\fL\in\Bnd_d(\Cluster)$, we choose an appropriate \emph{test set}
$\fT\subset\Orb\sminus\Cluster$ (typically, also a union of orbits) and use
the same techniques as in \autoref{ss.max.ext} to
compute the set
\[*
\fT_\mu:=\bigl\{\orb\in\fT\bigm|
 \mbox{$\fL'\cap\orb=\bnd(\orb)$ for some
 extension $\fL'\supset\fL$ satisfying~\eqref{eq.fixed.ext}}\bigr\}.
\]
Then, if
\[*
\bigl(\bnd(\Cluster)-\ls|\fL\cap\Cluster|\bigr)
 +\sum\bigl(\bnd(\orb)-\bnd'(\orb)\bigr)\ge\bnd(\Orb)-\goal,\qquad
 \orb\in\fT\sminus\fT_\mu,
\]
see~\eqref{eq.values},
we conclude that $\fL$ has no extensions
satisfying~\eqref{eq.goal} and~\eqref{eq.fixed.ext}.

\section{Algorithms}\label{algorithms}

In the first four sections, we describe a rough estimate on the bounds
$\bnd(\orb)$ in the Niemeier lattices with many roots (and, hence, large
combinatorial orbits). Then, in \autoref{s.Bset}, we explain the algorithms used to
compute the sets $\Bnd_d(\Cluster)$ in those few cases where the rough
estimates do not suffice.

\subsection{Bounds \via\ blocks\pdfstr{}{
 {\rm(see~\cite{degt:sextics})}}}\label{s.blocks}
Let $\bN:=\N(\bR)=\N(\bigoplus_k\bR_k)$ be a Niemeier lattice
rationally generated by roots, where the \emph{blocks} $D_k$, $k\in\Omega$,
are the irreducible
components of the \emph{maximal} root system $\bR\subset\bN$ and $\Omega$ is the
index set.
Thus, we have $\bN\subset\bR\dual=\bigoplus_k\bR_k\dual$; the vectors in
$\discr\bR=\bR\dual\!/\bR=\bigoplus_k\discr\bR_k$ that are declared
``integral'' are described in~\cite{Conway.Sloane}.
(We also use the convention of~\cite{Conway.Sloane} for the numbering of the
discriminant classes of irreducible root systems.)

Let $|_k\:\bN\to\bR_k\dual$ be the orthogonal projection. For a vector
$v\in\bN$, we often abbreviate $v_k:=v|_k$, so that $v=\sum_kv_k$,
$v_k\in\bR_k\dual$. Define the \emph{support}
\[*
\supp v:=\bigl\{k\in\Omega\bigm|v_k\ne0\bigl\}.
\]
The group $\RG_\hbar(\bN)$ preserves each block~$\bR_k$ and, hence, we can
also speak
about the \emph{support} $\supp\orb\subset\Omega$ of a combinatorial
orbit~$\orb$. (It is worth mentioning that, for each $k\in\Omega$,
the squares $l^2,l_k^2\in\Q$,
products
$l\cdot\hbar,l_k\cdot\hbar_k\in\Q$, and discriminant classes
$l\bmod\bR\in\discr\bR$ and $l_k\bmod\bR_k\in\discr\bR_k$,
$l\in\orb$, are also constant within each combinatorial orbit~$\orb$.)

Fix a combinatorial orbit~$\orb$ and define the \emph{count} and \emph{bound}
of a block~$\bR_k$ \via
\[*
\cnt(\bR_k):=\ls|\orb|_k|,\qquad \bnd(\bR_k):=\max\ls|\frak{R}|,
\]
where $\frak{R}\subset\orb|_k$ is a subset satisfying the condition
\[
\text{for $l',l''\in\frak{R}$, one has $l^{\prime2}-l'\cdot l''=0$ (iff $l'=l''$),
 $2$, $3$, or $4$}.
\label{eq.intersection.block}
\]
In other words, we bound the cardinality of subsets $\fL\subset\orb$
satisfying~\eqref{eq.weaker_property} and such that all planes $l\in\fL$ have the
same fixed restriction to all other blocks $\bR_s\ne\bR_k$.
Then, we have (\cf. \cite{degt:sextics})
\[
\cnt(\orb)=\prod_k\cnt(\bR_k),\qquad
\bnd(\orb)\le\cnt(\orb)\min_k\frac{\bnd(\bR_k)}{\cnt(\bR_k)}.
\label{eq.bound}
\]
For smaller blocks ($\bA_{\le7}$, $\bD_{\le6}$, and most $\bE$-type blocks)
the individual counts $\cnt(\bR_k)$ and bounds $\bnd(\bR_k)$ are computed by
brute force, and the resulting estimates~\eqref{eq.bound} suite most our
needs. For larger blocks, we use even more rough estimates, based on the
standard representation of the $\bA$- and $\bD$-type root systems as
sublattices of the odd unimodular lattice
\[*
\bH_{n}:=\bigoplus\Z\e_i,
\quad e^2_i=1,
\quad i\in\IS:=\{1,\ldots,n\}.
\]
(When working with this lattice, we let $\vv{o}:=\sum_{i\in o}\e_i$
for a subset $o\subset\IS$.)
Then, given a vector $\hbar_k=\sum_i\Ga_ie_i\in\bH_n\otimes\Q$,
we subdivide the block $\bR_k\dual\subset\bH_n\otimes\Q$ into ``subblocks''
\[*
\textstyle
\bR_k(\Ga):=\bigl\{\sum_i\Gb_ie_i\in\bH_n\otimes\Q\bigm|i\in\supp(\Ga)\bigr\},
\quad\supp(\Ga):=\bigl\{i\in\IS\bigm|\Ga_i=\Ga\bigr\},
\]
on which $\hbar_k$ is constant.
We obtain combinatorial counts and bounds, in the sense of~\eqref{eq.intersection.block},
for each subblock and use an obvious analogue of~\eqref{eq.bound} to
estimate $\bnd(\bR_k)$. The technical details are outlined in the
next few sections. We use, without further references, the following simple
observation.

\lemma\label{lem.shortest}
If $l\in\fF(\hbar)$ and $l\bmod\bR\ne0\in\discr\bR$,
then each $l_k$, $k\in\Omega$, is a shortest vector in its discriminant class
$l_k\bmod\bR_k\in\discr\bR_k$.
\endlemma

\proof
Indeed, otherwise the \emph{nontrivial} discriminant class $l\bmod\bR$ would
contain a shorter vector~$v$, necessarily of square $l^2-2=2$, contradicting
to the assumption that all roots in~$\N$ are in~$\bR$.
\endproof

\subsection{Root systems $\bA\sb{n}$}\label{s.An}
A block~$\bR_k$ of type $\bA_n$ is $\vv{\IS}^\perp\subset\bH_{n+1}$:
\[*
\textstyle
\bA_n=\bigl\{\sum_i\Ga_i\e_i\in\bH_{n+1}\bigm|\sum_i\Ga_i=0\bigr\}.
\]
One has $\discr\bA_n=\Z/(n+1)$, with a generator of square $n/(n+1)\bmod2\Z$,
and the shortest representatives of the
discriminant classes are vectors of the form
\[*
\be_{o}:=\dfrac1{n+1}
 \bigl(\ls|\bar o|\vv{o}-\ls|o|\vv{\bar o}\bigr),\qquad
 \be_{o}^2=\dfrac{\ls|o|\ls|\bar o|}{n+1},
\]
where $o\subset\IS$ and $\bar o$ is the complement. We have
$\be_{\bar o}=-\be_o$ and
\[*
\be_{r}\cdot\be_{s}=\ls|r\cap s|-\frac{\ls|r|\ls|s|}{n+1}.
\]
In particular, if $\ls|r|=\ls|s|$, or, equivalently,
$\e_r$ and $\e_s$ are in the same
discriminant class, then
\[
\be_{r}^2-\be_{r}\cdot\be_{s}=\frac12\ls|r\sdif s|,
\label{eq.A.intr}
\]
where $\sdif$ is the symmetric difference. Hence, in the case when $l_k$ is a
shortest vector in its (nonzero) discriminant class, the bound
$\bnd(\bR_k(\Ga))$
can be estimated by the following lemma, applied to $\Sigma=\supp(\Ga)$.

\lemma\label{lem.max.set.A}
Consider a finite set~$\Sigma$,\mnote{DgIt: $\Sigma$ instead of $S$}
$\ls|\Sigma|=n$, and let $\fS$ be a collection of
subsets $s \subset \Sigma$ with the following properties\rom:
\roster
\item\label{sets.m}
all subsets $s\in\fS$ have the same fixed cardinality~$m$\rom;
\item\label{sets.dif}
if $r,s\in\fS$, then $\ls|r\sdif s|\in\{0, 4, 6, 8\}$.
\endroster
Then, for small $(n,m)$, the maximum $\CA_{m,n}:=\max\ls|\fS|$
is as follows\rom:
\[*
\minitab\
(n,m):&    (n,1)&(n,2)             &(6,3)&(7,3)&(8,3)&(9,3)&(8,4)\cr
\CA_{m,n}:&    1&\lfloor n/2\rfloor&    4&    7&    8&   12&   14\cr
\endminitab
\]
More generally,
\[*
\CA_{3,n}\le\left\lfloor\frac{n}3\Bigl\lfloor\frac{n-1}2\Bigr\rfloor\right\rfloor;\qquad
\CA_{m,n}\le\left\lfloor\frac1m\binom{n}{m - 1}\right\rfloor\quad\mbox{for $m\ge1$}.
\]
\endlemma

Note that, if a collection~$\fS$ is as in the lemma, then so
is the collection $\{\bar s\,|\,s\in\fS\}$.
Hence, $\CA_{m,n}=\CA_{n-m,n}$ and we can always assume that $2m\le n$.

\proof[Proof of \autoref{lem.max.set.A}]
The first two values are obvious;
the others are obtained by listing all admissible
collections.

The general estimate for $m=3$ follows from the observation that any two
subsets in~$\fS$ have at most one common point and, hence, each point of~$\Sigma$ is
contained in at most $\bigl\lfloor\frac12(n-1)\bigr\rfloor$ such subsets.

For the last bound, we merely observe that each $(m-1)$-element set
$r \subset \Sigma$ is contained in at most one set $s\in\fS$.
\endproof

There remains to consider a subblock $\bR_k(\Ga)$ of a block~$\bR_k$
containing vectors of the form $l_k=\vv{r}-\vv{s}$, where
$r,s\subset\IS$, $r\cap s=\varnothing$, and $\ls|r|=\ls|s|=1$ or~$2$.
In the latter case, one must have $l_k\cdot\hbar_k=4$, and it follows that
$\ls|(r\cup s)\cap\supp(\Ga)|\le3$ for each $\Ga\in\Q$.
Letting $r_\Ga:=r\cap\supp(\Ga)$, $s_\Ga:=s\cap\supp(\Ga)$,
the bounds are as follows:
\roster
\item\label{bound.An.1}
if $\ls|r_\Ga|+\ls|s_\Ga|=1$, then, obviously, $\bnd(\bR_k(\Ga))=1$;
\item\label{bound.An.2}
if $(\ls|r_\Ga|,\ls|s_\Ga|)=(2,0)$ or $(0,2)$, then
distinct sets $r_\Ga$ (respectively, $s_\Ga$)
must be
pairwise disjoint and, hence,
$\bnd(\bR_k(\Ga))=\bigl\lfloor\frac12\ls|\supp(\Ga)|\bigr\rfloor$;
\item\label{bound.An.roots}
if $\ls|r_\Ga|=\ls|s_\Ga|=1$,
then distinct sets $r_\Ga$ must also be pairwise disjoint and, hence,
$\bnd(\bR_k(\Ga))=\ls|\supp(\Ga)|$;
\item\label{bound.An.three}
if $(\ls|r_\Ga|,\ls|s_\Ga|)=(2,1)$ or $(1,2)$, then
$\bnd(\bR_k(\Ga))=\ls|\supp(\Ga)|=3$.
\endroster

\subsection{Root systems $\bD\sb{n}$}\label{s.Dn}
A block~$\bR_k$ of type~$\bD_n$
can be defined as the maximal even sublattice in~$\bH_n$:
\[*
\textstyle
\bD_n=\bigl\{\sum_i\Ga_i\e_i\in\bH_{n}\bigm|\sum_i\Ga_i=0\bmod2\bigr\}.
\label{eq.Dn}
\]
If $n\ge5$, the group $\OG(\bD_n)$ is an index~$2$ extension of~$\RG(\bD_n)$:
it is
generated by the reflection against the hyperplane orthogonal to any
of~$\e_i$. Hence, up to $\OG(\bD_n)$, we can assume that, in the expression
$\hbar_k=\sum_i\Ga_i\e_i$, all coefficients $\Ga_i\ge0$. We always make this
assumption (and adjust the results afterwards) when describing the orbits and
computing counts and bounds, as otherwise the description of combinatorial
orbits is not quite combinatorial.

One has $\discr\bD_n=\Z/2\oplus\Z/2$ (if $n$ is even) or $\Z/4$ (if $n$ is
odd); the shortest vectors are
\[*
\e_i,\ i\in\IS,\quad\text{and}\quad
 \be_o:=\frac12(\vv{o}-\vv{\bar o}),\
 o\subset\IS,\quad
 \be_o^2=\frac{n}4
\]
(the class $\be_o\bmod\bD_n$ depends on the parity of~$\ls|o|$)
and we have a literal analogue of~\eqref{eq.A.intr} for any pair
$r,s\subset\IS$.
Thus, if $\bR_k\ni\be_o$,
the bounds $\bnd(\bR_k(\Ga))$ are estimated by
\autoref{lem.max.set.A} (if $\Ga\ne0$) or
\autoref{lem.max.set.D} below (if $\Ga=0$), applied to $\Sigma=\supp(\Ga)$.

\lemma\label{lem.max.set.D}
The maximal cardinality
of a collection~$\fS$ satisfying condition~\iref{sets.dif}
of \autoref{lem.max.set.A}
is bounded \via\
\[*
\ls|\fS|\le\max_{m\ge0}
 \bigl(\CA_{m,n}+\CA_{m+2,n}+\CA_{m+4,n}+\CA_{m+6,n}+\CA_{m+8,n}\bigr),
\]
where
$\CA_{m,n}$ is as in \autoref{lem.max.set.A} and we let $\CA_{m,n}=0$
unless $0\le m\le n$.
\endlemma

\proof
It suffices to observe that all sets $s\in\fS$ have cardinality of the same
parity and that $\bigl|\ls|s|-\ls|r|\bigr|\le8$ for any pair $r,s\in\fS$.
\endproof

The few remaining cases are listed below.
\roster
\item\label{bound.Dn.2e}
If $\bR_k(\Ga)\ni\pm2e_i$, $i\in\supp(\Ga)$, then
$\bnd(\bR_k(\Ga))=\ls|\supp(\Ga)|$.
\endroster
Assume that
$l_k=\sum(\pm e_i)$, $i\in o\subset\CS$, $\ls|o|\le4$.
If $\Ga=0$, then
\roster[\lastitem]
\item\label{bound.Dn.zero}
$\ls|o\cap\supp(\Ga)|=0$, $1$, $2$ and
$\bnd(\bR_k(\Ga))\le1$, $2$,
$4\bigl\lfloor\frac12\ls|\supp(\Ga)|\bigr\rfloor$, respectively,
\endroster
similar to \autoref{s.An}.
(Here, the last number is a bound on the size of
a union of (affine) Dynkin diagrams admitting an
isometry to $\bD_{\ls|\supp(\Ga)|}$, \cf. \autoref{ss.integral.roots} below.)
If $\Ga\ne0$, the numbers of
signs $\pm$ within $\supp(\Ga)$ are also fixed, and the options are as
follows:
\roster[\lastitem]
\item\label{bound.Dn.same}
$m:=\ls|o\cap\supp(\Ga)|\le3$ and all signs are the same:
by an analogue of~\eqref{eq.A.intr},
a bound
on $\bnd(\bR_k(\Ga))$ is given by \autoref{lem.max.set.A} applied to
$\Sigma=\supp(\Ga)$;
\item\label{bound.Dn.roots}
$\ls|o\cap\supp(\Ga)|=2$ and the signs differ:
$\bnd(\bR_k(\Ga))=\ls|\supp(\Ga)|$ as in \autoref{s.An}\iref{bound.An.roots};
\item\label{bound.Dn.three}
$\ls|o\cap\supp(\Ga)|=3$ and the signs differ:
$\bnd(\bR_k(\Ga))=\ls|\supp(\Ga)|=3$.
\endroster

\subsection{Other root lattices}\label{s.other}
We use the description of \autoref{s.An} for the blocks $\bA_n$, $n\ge8$, and
that of \autoref{s.Dn}, for $\bD_n$, $n\ge7$. Below we outline a few more
tricks that simplify the computations for other blocks, mainly those of
type~$\bE_8$.

\subsubsection{Integral roots}\label{ss.integral.roots}
Assume that $\hbar_k=0$ and $l_k$ is an integral root, \ie, $l_k^2=2$ and
$l_k\in\bR_k$.
By~\eqref{eq.intersection.block}, the pairwise products of roots take values
in $\{0,-1,-2\}$, \ie,
an admissible subset in $\bR_k$ is a union~$\Gamma$ of (affine) Dynkin
diagrams (including those of type \smash{$\tA_1$})
that is mapped isometrically to~$\bR_k$.
If $\bR_k\cong\bE_8$, $\bE_7$, or $\bD_{\text{even}}$, \cf.
\autoref{s.Dn}\iref{bound.Dn.zero}, then $\bR_k$ is rationally generated by
pairwise orthogonal roots and, hence, the maximal union~$\Gamma$ as above is
$n\tA_1$, $n=\rank\bR_k$. Thus, we have $\bnd(\bR_k)=2n$ in this case.

If $\bR_k\cong\bE_8$ and $\hbar_k,l_k\in\bR_k$ are orthogonal integral roots,
$\hbar_k^2=l_k^2=2$, $l_k\cdot\hbar_k=0$, we can apply the above argument to
$\hbar_k^\perp\cong\bE_7$ to obtain $\bnd(\bR_k)=14$.
If $l_k\cdot\hbar_k=1$,
the bound is $\bnd(\bR_k)\le4$, which is
the maximal valency of a vertex in (any) affine
Dynkin diagram. If $l_k\cdot\hbar_k=2$, then $l_k=\hbar_k$ and
$\cnt(\bR_k)=\bnd(\bR_k)=1$.

\subsubsection{Planes contained in $\bE\sb8$}\label{ss.E8}
Assume that $\bR_k\cong\bE_8$ and $l_k^2=4$, so that necessarily $l=l_k$ and
$l_k\cdot\hbar_k=4$ (and, hence, $\hbar_k^2\ge4$).
Below we consider two cases where
a long computation can be simplified.

If $\hbar_k^2=12$, then $\hbar=\hbar_k$ and all planes~$l_k$ are in the
index~$2$ sublattice
\[*
\bigl\{x\in\bR_k\bigm|x\cdot\hbar_k=0\bmod2\bigr\}\cong\bD_8.
\]
Up to automorphism of $\bD_8$, we have $\hbar_k=2(\vv{o})$, $o=\{1,2,3\}$,
and $l_k=\vv{r}\pm\e_i\pm\e_j$, where $r\subset o$, $\ls|r|=2$, and
$i,j\in\CI\sminus o$. Arguing as in \autoref{s.Dn}, we obtain
$\bnd(\bR_k)\le24$.

If $\hbar_k^2=8$ and $\hbar_k\in2\bR_k$, then each plane has the form
$l_k=\frac12\hbar_k+r$, where $r$ is a root in $\hbar_k^\perp\cong\bE_7$.
Hence, we have $\bnd(\bR_k)=14$ as in \autoref{ss.integral.roots}.

\subsection{The computation of $\Bnd\sb{d}(\Cluster)$}\label{s.Bset}
Most sets $\Bnd_d(\Cluster)$ or $\pBnd_d(\Cluster)$
introduced in \autoref{s.orbits} are computed using
the so-called \emph{patterns}, as explained below in this section.

\subsubsection{Patterns\pdfstr{}{
 {\rm(see~\cite{degt:sextics})}}}\label{s.patterns}
The \emph{pattern} of a pseudo-geometric set $\fL$ is the function
\[*
\pat_\fL\:\Orb\to\NN,\quad \orb\mapsto\ls|\fL\cap\orb|\in\bset(\orb).
\]
To compute a set $\Bnd_d(\Cluster)$, we start with listing,
up to the action of the group $\stab\hbar$,
the abstract restricted patterns
$\pat\:\Cluster\to\NN$ satisfying the conditions
\[*
\pat(\orb)\in\bset(\orb), \qquad
\sum\pat(\orb)\ge\bnd(\Cluster)-d,\quad\orb\in\Cluster.
\]
Then, using the precomputed collections $\Bset(\orb)$, we build a
(pseudo-)geometric set $\fL$
orbit by orbit, as the completion of a union
$\bigcup_{\orb\in\Cluster}\fL_\orb$, $\fL_\orb\in\Bset(\orb)$,
$\ls|\fL_\orb\cap\orb|=\pat(\orb)$. In this computation, we
add combinatorial orbits $\orb\in\Cluster$ one by one,
use the group $\RG_\hbar(\bN)$ to reduce the overcounting,
and, at each step, make sure that the partial set~$\fL$
constructed is (pseudo-)geometric and that $\pat_\fL|_\Cluster=\pat$.
Further details are found in~\cite{degt:sextics}; we reuse the code developed
there.

\subsubsection{Clusters\pdfstr{}{
 {\rm(see~\cite{degt:sextics})}}}\label{s.clusters}
If the number $m(\borb_n)$ of combinatorial orbits in an orbit
$\borb_n=\Cluster$ is large, listing all abstract patterns in
\autoref{s.patterns} is difficult. In this case, we try to
subdivide~$\borb_n$ into \emph{clusters}~$\cluster_k$, not necessarily
disjoint, consisting of whole combinatorial orbits and such that $\stab\hbar$
induces an at least $1$-transitive action on the set of clusters. (The
construction of clusters, usually ``natural'', is described in the
respective proof sections case by case.) Then, we can compute abstract
patterns and construct (pseudo-)geometric sets cluster by cluster, assuming
the latter ordered by the decreasing of their ``complexity''
(see~\cite{degt:sextics} for more details).
Within each
cluster, we still use patterns
(compatible with the clusters already considered) and extend the
(pseudo-)geometric set orbit by orbit,
as in \autoref{s.patterns}.

\subsubsection{Meta-patterns}\label{s.meta}
In some cases, the number~$u$ of clusters is still large (requiring too many
steps in \autoref{s.clusters}), whereas each cluster is relatively small, so
that we can easily compute the set $\Bnd_*(\cluster_*)$ of
$\OG(\hbar)$-orbits of geometric sets pseudo-generated by a single cluster.
In these cases, we construct a geometric set
\[*
\fL=\fL_u\supset\fL_{u-1}\supset\ldots\supset\fL_0=\varnothing
\]
cluster by
cluster, using the full symmetry group $\OG_\hbar(\bN)$ and
a \emph{meta-pattern}, \ie, a set of values
\[*
\mpat:=\bigl\{\fT_1\ge\ldots\ge\fT_u\bigr\},\quad
 \fT_i\in\Bnd_*(\cluster_*),
\]
ordered lexicographically by the decreasing of the pair
$(\ls|\mval_i|,\ls|\fT_i|)$, $\mval_i\in\fT_i$, and such that
$\sum\ls|\mval_i|\ge\bnd(\Cluster)-d$. At each step~$k$, we pass from
$\fL_{k-1}$ to the sets
\[*
\fL_k:=\fF(\hbar)\cap\spn_2(\fL_{k-1}\cup\mval_k),
\]
taking for
$\mval_k\in\fT_k$ a single representative of each orbit of the
$\OG_\hbar(\bN)$-stabilizer of $\fL_{k-1}$; then, we select those sets~$\fL_k$ that
are pseudo-geometric and satisfy the condition
$\ls|\fL_k\cap\cluster_i|=\ls|\mval_i|$ for all $i=1,\ldots,k$.
This algorithm is similar to \autoref{s.patterns} (\cf.~\cite{degt:sextics}),
except that clusters are used instead of combinatorial orbits and the full
symmetry group $\OG_\hbar(\bN)$ is used instead of $\RG(\bN)$: at the
beginning, preparing the meta-patterns,
we do not associate the values $\fT_k$ with particular
clusters~$\cluster_k$.

\subsubsection{Iterated maximal subsets}\label{s.maximal}
Let $\Cluster$ be a complete \emph{admissible} set. (Typically, we take for
$\Cluster$ a union of orbits.) Clearly, any subset $\fL\subset\Cluster$ is also
admissible. If $\fL$ is also complete, then either
$\spn_2\fL\subset\spn_2\Cluster$ has positive corank or
$\spn_2\Cluster/\!\spn_2\fL$ is a $2$-group, nontrivial if $\fL$ is proper.
It follows that any maximal (with respect to inclusion)
proper complete subset
$\fL\subset\Cluster$ is of the form
\[*
\fL=\Cluster\cap\Ker v,\qquad
v\in\Hom(\spn_2\Cluster,\F_2)=\spn_2\Cluster/2\spn_2\Cluster,
\]
and these sets can easily be listed. (Usually, we also take into account the
action on $\spn_2\Cluster/2\spn_2\Cluster$ of the $\OG_\hbar(\bN)$-stabilizer
of~$\Cluster$.
Note that we do not assert that $\spn_2\Cluster/\!\spn_2\fL=\Z/2$.)
Iterating this procedure (and eliminating repetitions at the
intermediate steps), we can list all complete subsets $\fL\subset\Cluster$
and, in particular, compute the sets
\smash{$\Bnd_d(\Cluster)\subset\pBnd_d(\Cluster)$}.

\section{Lattices with few components}\label{S.few}

In this section we consider the Niemeier lattices with at most eight irreducible root
components and prove the following theorem.

\theorem\label{th:few_components}
Let $N$ be a Niemeier lattice other than $\N(12\bA_2)$, $\N(24\bA_1)$ or $\Lambda$.
Then, for any $\hbar \in N$, $\hbar^2 = 12$, and any geometric set $\fL \subset \fF(\hbar)$,
one has $\ls|\fL| \leq 350$.
\endtheorem

\proof
For each $12$-polarized Niemeier lattice $\bN\ni\hbar$, we list all
$\OG_\hbar(\bN)$-orbits $\borb_n$ and compute the number
$m(\borb_n)$ of combinatorial orbits $\orb\subset\borb_n$, the count
$\cnt(\orb)$, and the na\"{\i}ve bound on $\ls|\fL\cap\orb|$ given
by~\eqref{eq.bound}. Sometimes, this bound is improved by a brute force
computation; the best bound obtained is denoted by $\bnd(\orb)$.
The results are listed in several tables below,
where the bounds $\bnd(\orb)$ confirmed by brute force are underlined.

\convention\label{conv.tables}
{\4Listed in the tables are the isomorphism classes of vectors~$\hbar$ (the
rows with bold index in the first column) and, for each~$\hbar$, the orbits of
conics. The columns are as follows:
\roster*
\item
index for further references,
\item
the vector $\hbar$ or conic, one entry for each indecomposable component of
the maximal root system (see the next paragraph for the notation),
\item
the number of combinatorial orbits in each orbit,
\item
the line count, total (for $\hbar$) or per one combinatorial orbit, and
\item
the na\"\i ve bound (also total or per one combinatorial orbit).
\endroster
}

For the components $\hbar_k\in\bR_k\dual$ of~$\hbar$ we use the notation
$\hvec{\hbar_k^2}_d$, where $d$ is either the discriminant class
of~$\hbar_k^2$ (in the notation of~\cite{Conway.Sloane}) or,
if $\hbar_k\in\bR_k$, the
symbol
\[*
\mbox{$0$ (if $\hbar_k=0$),\quad $\circ$ (if $\hbar_k^2=2$),\quad
$\bullet$ (if $\hbar_k^2=4$),\quad $*$ (if $\hbar_k^2=6$)}.
\]
For the components $l_k$ of a plane~$l$, we use the notation
$\lvec{l_k\cdot\hbar_k}_d$, where $d$ has
the same meaning as for~$\hbar$. (By \autoref{lem.shortest}, $l_k^2$ is
uniquely determined by the subscript.)
\emph{In the cases considered}, these data determine $(\hbar_k,l_k)$ up to
$\RG(\bR_k)$.
(The occasional superscripts are artifacts left over from the
complete set of tables.)

Also shown in the tables is the na\"{\i}ve \latin{a priori} estimate $\bnd(\Orb)$
given by~\eqref{eq.naive}.
For the vast majority of lattices $\bN\ni\hbar$ we
have
$\bnd(\Orb)\le \goal$\iffullversion. \else, and these pairs are omitted.
\fi
The few cases where $\bnd(\Orb) \ge \goal$
are shown in bold, and we treat them separately.
\iffullversion\else
In the ``trivial'' cases equivalent to those already considered in another
lattice~$\bN$,
we usually also omit the list of orbits~$\borb_n$.
\fi
\endconvention


\input{\listdirectory/\prefix_D24}
\input{\listdirectory/\prefix_D16_E8}
\input{\listdirectory/\prefix_3E8}
\input{\listdirectory/\prefix_A24}
\input{\listdirectory/\prefix_2D12}
\input{\listdirectory/\prefix_A17_E7}
\input{\listdirectory/\prefix_D10_2E7}
\input{\listdirectory/\prefix_A15_D9}
\input{\listdirectory/\prefix_3D8}
\input{\listdirectory/\prefix_2A12}
\input{\listdirectory/\prefix_A11_D7_E6}
\input{\listdirectory/\prefix_4E6}
\input{\listdirectory/\prefix_2A9_D6}
\input{\listdirectory/\prefix_4D6}
\input{\listdirectory/\prefix_3A8}
\input{\listdirectory/\prefix_2A7_2D5}
\input{\listdirectory/\prefix_4A6}
\input{\listdirectory/\prefix_4A5_D4}

\input{\listdirectory/\prefix_6D4}

\def\porb#1{\obj2{\cluster_{#1}}}
\subsubsection{Configuration~\hlink{6D4}{1}}\label{s.6D4-1}
We subdivide the orbit $\oorb2$ into six clusters (not disjoint)
\[*
\porb{k}:=\bigl\{\orb\subset\oorb2\bigm|k\notin\supp\orb\bigr\},
\quad k\in\Omega,
\]
and use these clusters (see \autoref{s.clusters}) to compute
$\pBnd_{24}(\oorb2)=\varnothing$.


\input{\listdirectory/\prefix_6A4}

\input{\listdirectory/\prefix_8A3}

\def\porb#1{\obj6{\cluster_{#1}}}
\subsubsection{Configuration~\hlink{8A3}{1}}\label{s.8A3-1}
Denote $\CK:=\bigl\{k\in\Omega\bigm|\hh_k^2=1\bigr\}$, and
subdivide~$\oorb6$ into seven pairwise disjoint clusters
\[*
\porb{k}:=\bigl\{\orb\subset\oorb6\bigm|
 \mbox{$l_k^2=1$ for $l\in\orb$}\bigr\},
\quad k\in\CK.
\]
Using these clusters (see \autoref{s.clusters}),
we show that $\pBnd_{56}(\oorb6)=\varnothing$.


\def\porb#1{\obj6{\cluster_{#1}}}
\subsubsection{Configuration~\hlink{8A3}{2}}\label{s.8A3-2}
Denote $\CK:=\bigl\{k\in\Omega\bigm|\mbox{$\hh_k$ is a root}\bigr\}$, and
subdivide~$\oorb6$ into four pairwise disjoint clusters
\[*
\porb{k}:=\bigl\{\orb\subset\oorb6\bigm|k\notin\supp\orb\bigr\},
\quad k\in\CK.
\]
Using these clusters (see \autoref{s.clusters}),
we show that $\pBnd_{24}(\oorb6)=\varnothing$.
\endproof


\section{The lattice $N(12\bA\sb2)$}\label{S.12A2}

The goal of this section is the following theorem.

\theorem\label{th:12A_2}
For any $\hbar \in \N(12\bA_2)$, $\hbar^2 = 12$, and any geometric set $\fL \subset \fF(\hbar)$,
one has $\ls|\fL| \leq 350$.
\endtheorem

\proof
We proceed as in the previous section, considering $\OG(\bN)$-orbits of
square~$12$ vectors $\hbar\in\N(12\bA_2)$ one by one.
\def\ltitle{}%
\def\ltext{%
 There are \lcount{\thelattice} orbits;
 $\bnd(\fF)\le\lmax{\thelattice}$%
 \iffullversion.\else\ (see \autoref{tab.\thelattice}).\fi
}%
{\def\vectorbox#1{#1\,\,}\def\countskip{\ }\input{\listdirectory/\prefix_12A2}}

\def\porb#1{\obj3{\cluster_{#1}}}
\subsection{Configuration~\hlink{12A2}{1}}\label{s.12A2-1}
Let $\CK:=\bigl\{k\in\Omega\bigm|\hh_k^2=\frac23\bigr\}$, and
subdivide~$\oorb3$ into eleven clusters (not disjoint)
\[*
\porb{k}:=\bigl\{\orb\subset\oorb2\bigm|l_k\cdot\hh_k=\tfrac23\bigr\},
\quad k\in\CK,
\]
to show (see \autoref{s.clusters}) that $\pBnd_{56}(\oorb3)=\varnothing$.


\def\porb#1{\obj2{\cluster_{#1}}}
\subsection{Configuration~\hlink{12A2}{3}}\label{s.12A2-3}
Denote $\CK:=\bigl\{k\in\Omega\bigm|\mbox{$\hh_k$ is a root}\bigr\}$, and
subdivide~$\oorb2$ into three pairwise disjoint clusters
\[*
\porb{k}:=\bigl\{\orb\subset\oorb2\bigm|k\notin\supp\orb\bigr\},
\quad k\in\CK.
\]
Using these clusters (see \autoref{s.clusters}),
we show that $\pBnd_{20}(\oorb2)=\varnothing$.
More precisely, we find eight sets $\fL\subset\pBnd_6(\porb1)$, all of size
$\ls|\fL|=72$. Seven sets are maximal, see \autoref{ss.max.set}.
The eighth one is
of rank $20$, and adding a single maximal orbit (see \autoref{ss.max.ext})
from another cluster produces
a single maximal set of size~$156$.

A direct computation using patterns (see \autoref{s.patterns}) shows that there are
four sets $\fL\in\pBnd_7(\oorb5)$; they are all maximal
(see \autoref{ss.max.set}) and of size $\ls|\fL|=44$.

Finally,
let $\Cluster:=\oorb4\cup\oorb6$; this set is complete and admissible.
Listing iterated maximal subsets (see \autoref{s.maximal}),
we find $19$ sets
$\fL\in\Bnd_{13}(\Cluster)$.
All but one are ruled out by counting maximal orbits (see \autoref{ss.max.orbit})
with the test set $\fT=\oorb2$.
The exception is $\fL=\oorb4$, and in this case we use patterns (see
\autoref{s.patterns}) to show that, for any extension
$\fL'\supset\fL$ satisfying~\eqref{eq.fixed.ext}, we have
\roster*
\item
$\ls|\fL'\cap\borb|\le\bnd(\borb)-4$ for $\borb=\oorb1$, $\oorb3$ or any of
the three clusters $\porb{k}\subset\oorb2$, and
\item
$\ls|\fL'\cap\oorb5| \le \bnd(\oorb5)-33$.
\endroster

\remark\label{rem.12A2-3}
The sublattice $\spn_2\Cluster\subset\bN$ has no proper mild extensions, and its
$2$-discriminant violates the condition stated in
\autoref{Niemeier_back}\iref{cond3}.
Therefore, when computing $\Bnd_{13}(\Cluster)$, we confine ourselves to
subsets $\fL\subset\Cluster$ of corank at least one. The rest of the argument
applies to pseudo-geometric sets as well.
\endremark


\def\porb#1{\obj5{\cluster_{#1}}}
\subsection{Configuration~\hlink{12A2}{4}}\label{s.12A2-4}
Denote $\CK:=\bigl\{k\in\Omega\bigm|\hh_k^2=\frac83\bigr\}$, and
subdivide~$\oorb5$ into two disjoint clusters
\[*
\porb{k}:=\bigl\{\orb\subset\oorb5\bigm|
 \mbox{$l_k\cdot\hh_k=\tfrac43$ for $l\in\orb$}\bigr\},
\quad k\in\CK,
\]
to show (see \autoref{s.clusters}) that $\Bnd_{25}(\oorb5)=\varnothing$.
On the other hand, one has
\[*
\Bnd_0(\oorb1\cup\oorb3\cup\oorb4\cup\oorb7\cup\oorb8\cup\oorb9)=
\Bnd_0(\oorb2)=\Bnd_0(\oorb6)=\varnothing
\]
directly as in \autoref{s.patterns}. (Here, we do use condition~\iref{cond3} in
\autoref{Niemeier_back}.)


\subsection{Configuration~\hlink{12A2}{5}}\label{s.12A2-5}
We have $\Bnd_0(\Orb)=\varnothing$ directly as in \autoref{s.patterns}.
\endproof


\section{The lattice $N(24\bA\sb1)$}

The goal of this section is the following theorem.

\newcs{check-24A1-7}{\noexpand\hidelines}
\newcs{check-24A1-8}{\noexpand\hidelines}
\newcs{check-24A1-9}{\noexpand\hidelines}
\newcs{check-24A1-11}{\noexpand\hidelines}
\newcs{check-24A1-12}{\noexpand\hidelines}
\makeatletter
\let\@int\@@int
\let\@rat\@@rat
\makeatother

\theorem\label{th:24A_1}
Let $\hbar \in \N(24\bA_1)$, $\hbar^2 = 12$, and let $\fL \subset \fF(\hbar)$
be a geometric set. Then, unless
\roster*
\item $\ls|\fL| = 357$ and $\fL = \Lsub1$, see \eqref{eq.Lsub.1},
\endroster
one has $\ls|\fL| \leq 350$.
\endtheorem

\proof
We proceed as in the previous sections.
Each component $v_k\in\bR_k\dual$,
$k\in\Omega$, of a vector $v\in\bN$
is a multiple of the generator $r_k\in\bR_k$.
To save space, we use the following notation for the coefficient~$\Ga$ in
$v_k=\Ga r_k$:
\[*
\mbox{$\Azero$ ($\Ga=0$),\quad
$\Ahalf$ or $\Aminus$ ($\Ga=\pm\frac12$),\quad
$\Aroot$ ($\Ga=\pm1$),\quad
$\Aplus$ ($\Ga=\pm\frac32$),\quad
$\Art$ ($\Ga=\pm2$)}.
\]
Here, $\Aminus$ is used only for~$l_k$
and only if $\hbar_k\cdot l_k<0$; in all other cases, the
signs of~$l_k$ and $\hbar_k$ agree, so that we have $\hbar_k\cdot l_k\ge0$.

\def\thelattice{24A1}
There are \lcount{\thelattice} orbits $\hbar\in\N(24\bA_1)$,
and $\bnd(\fF)\le\lmax{\thelattice}$%
\iffullversion.\else\ (see \autoref{tab.24A1}).\fi

{\def\latticetext{}\input{\listdirectory/\prefix_24A1}}

Fix a basis $\{r_k\}$, $k\in\Omega$, for $24\bA_1$ consisting of roots.
The kernel
\[*
\bN\bmod24\bA_2\subset\discr24\bA_1\cong(\Z/2)^{24}
\]
of the extension is the Golay code $\CC_{24}$
(see~\cite{Conway.Sloane}).
The map $\supp$ identifies codewords with
subsets of $\Omega$; then, $\CC_{24}$ is invariant
under complement and, in addition to~$\varnothing$ and~$\Omega$,
it consists of $759$ octads, $759$ complements thereof,
and $2576$ dodecads.

To simplify the notation, we identify the basis vectors $r_k$
(assumed fixed) with their indices $k\in\Omega$. For a subset
$\CS\subset\Omega$, we let $\vv\CS:=\sum r$, $r\in\CS$,
and abbreviate $\cw\CS:=\frac12\vv\CS\in\bN$ if $\CS\in\CC_{24}$ is a codeword.

\subsection{Configuration~\hlink{24A1}{1}}\label{s.24A1-1}
We have $\stab\hh=M_{24}$ (the Mathieu group, which is the symmetry group of
the Golay code) and $\hh=\cw\Omega$.
The set $\fF(\hh)=\oorb1$ is complete and admissible;
listing iterated maximal subsets (see \autoref{s.maximal}),
after four steps of the algorithm we find out
that there are no complete
subsets $\fL\subset\fF(\hh)$ of rank $\rank\fL\le21$ and
size $\ls|\fL|\ge385$ (\cf. \autoref{s.Leech-1} below).

\def\porb#1{\obj3{\cluster_{#1}}}

\subsection{Configuration~\hlink{24A1}{2}}\label{s.24A1-2}
We have $\ls|\stab\hh|=11520$ and
$\hh=\vv\CR$, where $\ls|\CR|=6$ and $\CR$ is a subset of an octad
(\cf. \autoref{s.24A1-11}).
We subdivide $\oorb3$ into $15$ pairwise disjoint clusters
\[*
\porb{o}:=\bigl\{\orb\subset\oorb3\bigm|\supp\orb\supset o\bigr\},
\quad o\subset\CR,\ \ls|o|=4,
\]
and use these clusters (see \autoref{s.clusters}) to
show that $\Bnd_{156}(\oorb3)=\varnothing$.
(The set $\bset(\orb)$, $\orb\subset\oorb3$,
equals $\{0,\ldots,4,6,8\}$, which simplifies the computation.)


\def\porb#1{\obj3{\cluster_{#1}}}

\subsection{Configuration~\hlink{24A1}{3}}\label{s.24A1-3}
We have $\ls|\stab\hh|=11520$ and $\hh=\cw\CO+\vv\CR$,
where $\CO$ is an octad and $\CR\subset\CO$ a $2$-element set.
Subdivide $\oorb3$ into $15$ pairwise disjoint clusters
\[*
\porb{o}:=\bigl\{\orb\subset\oorb3\bigm|\supp\orb\supset o\bigr\},
\quad o\subset\CO\sminus\CR,\ \ls|o|=2.
\]
and use these clusters (see \autoref{s.clusters}) to
show that $\Bnd_{156}(\oorb3)=\varnothing$.
(The set $\bset(\orb)$, $\orb\subset\oorb3$,
equals $\{0,\ldots,4,6,8\}$, which simplifies the computation.
In spite of the apparent similarity, this polarized lattice is not equivalent
to \config{24A1}2, see \autoref{s.24A1-2}.)


\def\porb#1{\obj3{\cluster_{#1}}}
\def\pporb#1{\obj2{\cluster'_{#1}}}

\subsection{Configuration~\hlink{24A1}{4}}\label{s.24A1-4}
We have $\ls|\stab\hh|=20160$ and $\hh=\cw\CO+r$, where $\CO\ni r$
is a codeword of length~$16$.
Let $\CK:=\Omega\sminus\CO$.

We subdivide the orbit~$\oorb3$ into $28$ pairwise disjoint clusters
\[*
\porb{k}:=\bigl\{\orb\subset\oorb3\bigm|\supp\orb\supset k\bigr\},
\quad k\subset\CK,\ \ls|k|=2,
\]
and, using these clusters and meta-patterns (see \autoref{s.meta}), show that
$\Bnd_{50}(\oorb3)=\varnothing$.
(Note that
$\ls|\fL\cap\porb{k}|\in\{0,\ldots,8,10,12\}$
for each~$k$ and each geometric set~$\fL$.)

Now, consider the complete admissible set
$\Cluster:=\oorb1\cup\oorb2\cup\oorb4$. Listing iterated maximal subsets
(see \autoref{s.maximal}), we find $34$ sets $\fL\in\Bnd_{70}(\Cluster)$, of
which
all but two are ruled out by counting maximal orbits (see
\autoref{ss.max.orbit}), using $\fT=\oorb3$ for the test set.

One of the two exceptions is a set $\fL$ of size $87$, for which
$\ls|\fT_\mu|=96$ (see \autoref{ss.max.orbit} for the notation).
Extending the pattern $\pat_\fL$ (see \autoref{ss.max.ext}) \via\
$\orb\mapsto\bnd(\orb)=2$ for each orbit $\orb\in\fT_\mu$
(and leaving the other values undefined)
and applying the algorithm of \autoref{s.patterns}, we show that
$\ls|\fL'|\le350$ for each extension $\fL'\supset\fL$
satisfying~\eqref{eq.fixed.ext}.

The other exception is $\fL=\oorb2$, with $\fT_\mu=\fT$.
Arguing as in \autoref{ss.max.ext} (extending $\pat_\fL$ by a single extra
value $\orb\mapsto1$ or $2$ for some orbit $\orb\subset\oorb3$ and analyzing
the output),
we observe that
$\ls|\fL'\cap\porb{k}|\in\{0,12\}$ for each cluster $\porb{k}$ and each
extension $\fL'\supset\fL$ satisfying~\eqref{eq.fixed.ext}.
Using this observation and meta-patterns (see \autoref{s.meta}) with
the set of values restricted accordingly,
we show that
$\ls|\fL'|\le350$ with the exception of a single,
up to $\OG_{\hh}(\bN)$, saturated
set~$\Lsub1$
of rank $21$ and size $357$.
This set is characterised
by any of the eight patterns
\[*
\pat_\cluster(\orb)=\begin{cases}
                      1, & \mbox{if } \orb\subset\oorb2, \\
                      2, & \mbox{if } \orb\subset\cluster, \\
                      0, & \mbox{otherwise},
                    \end{cases}
\]
where $\cluster:=\bigl\{\orb\subset\oorb3\bigm|\supp\orb\not\ni k\bigr\}$ for
some fixed $k\in\CK$. Alternatively,
\[
\Lsub1=\fF(\hh)\cap\spn(\cw\CK,\hh-4r,k)^\perp.
\sublabel1
\]




\def\porb#1{\obj3{\cluster_{#1}}}
\def\pporb#1{\obj2{\cluster'_{#1}}}
\subsection{Configuration~\hlink{24A1}{5}}\label{s.24A1-5}
We have $\ls|\stab\hh|=2304$ and $\hh=\cw\CO+\vv\CR$, where
$\CO$ is an octad and $\CR\subset\Omega\sminus\CO$ a
$4$-element set, so that there is an octad $o\supset\CR$ such that
$\ls|o\cap\CO|=4$ (\cf. \autoref{s.24A1-12}).
We subdivide $\oorb3$ into four pairwise disjoint
clusters
\[*
\porb{r}:=\bigl\{\orb\subset\oorb3\bigm|\supp\orb\not\ni r\bigr\},
\quad r\in\CR,
\]
and use these clusters (see \autoref{s.clusters}) to
show that $\Bnd_{60}(\oorb3)=\varnothing$.
Then, subdividing~$\oorb2$ into six pairwise disjoint clusters
\[*
\pporb{s}:=\bigl\{\orb\subset\oorb2\bigm|\supp\orb\supset s\bigr\},
\quad s\subset\CR,\ \ls|s|=2,
\]
we show (see \autoref{s.clusters}) that $\Bnd_{39}(\oorb2)=\varnothing$.


\def\porb#1{\obj4{\cluster_{#1}}}
\def\pporb#1{\obj2{\cluster'_{#1}}}

\subsection{Configuration~\hlink{24A1}{6}}\label{s.24A1-6}
We have
$\ls|\stab\hh|=11520$ and $\hh=\cw\CO+\vv\CR$,
where
$\CO$ is a codeword of length~$16$ and
$\CR\subset\Omega\sminus\CO$ is a $2$-element set.
Let $\CK:=\Omega\sminus(\CO\cup\CR)$.
We subdivide $\oorb4$ into six clusters (not disjoint)
\[*
\porb{k}:=\bigl\{\orb\subset\oorb4\bigm|\supp\orb\ni k\bigr\},
\quad k\in\CK,
\]
and use these clusters (see \autoref{s.clusters})
to show that $\Bnd_{35}(\oorb4)=\varnothing$.
Next, we subdivide $\oorb2$ into
twelve pairwise disjoint clusters
\[*
\pporb{s}:=\bigl\{\orb\subset\oorb2\bigm|\supp\orb\supset s\bigr\},
\quad s\in\CR\times\CK,
\]
and show (using \autoref{s.clusters})
that $\Bnd_{22}(\oorb2)=\varnothing$. Finally, listing iterated
maximal subsets (see \autoref{s.maximal})
in the complete admissible set
$\Cluster:=\oorb1\cup\oorb3\cup\oorb5$, we obtain
$\Bnd_{29}(\Cluster)=\varnothing$.




\def\porb#1{\obj3{\cluster_{#1}}}

\subsection{Configuration~\hlink{24A1}{7}}\label{s.24A1-7}
We have $\ls|\stab\hh|=336$ and $\hh=\cw\CO+\vv\CR+r$, where $\CO\ni r$ is an
octad and $\CR\subset\Omega\sminus\CO$ a $2$-element set.
This configuration is equivalent to \config{8A3}1 in $\N(8\bA_3)$,
see \autoref{s.8A3-1},
where we treat pseudo-geometric sets as well.

\subsection{Configuration~\hlink{24A1}{8}}\label{s.24A1-8}
We have $\ls|\stab\hh|=660$ and $\hh=\cw\CO+r+s$, where $\CO\ni r$ is a
dodecad and $s\in\Omega\sminus\CO$.
This configuration is equivalent to \config{12A2}1 in $\N(12\bA_2)$,
see \autoref{s.12A2-1}, where we treat pseudo-geometric sets as well.

\subsection{Configuration~\hlink{24A1}{9}}\label{s.24A1-9}
We have $\ls|\stab\hh|=432$ and $\hh=\cw\CO+\vv\CR$, where $\CO$ is a dodecad
and $\CR\subset\Omega\sminus\CO$ a $3$-element set.
This configuration is equivalent to \config{s.12A2}3 in $\N(12\bA_2)$,
see \autoref{s.12A2-3}, where we mainly treat pseudo-geometric sets. The only
exception, explained in \autoref{rem.12A2-3}, applies to the present case as
well, as
the sublattice $\spn_2(\borb_3\cup\borb_6)\subset\bN$
has no proper mild extensions.

\subsection{Configuration~\hlink{24A1}{10}}\label{s.24A1-10}
We have $\ls|\stab\hh|=40320$ and $\hh=\vv\CR+2r$,
where $\CR$ is a $2$-element set and $r\notin\CR$.
Using patterns (see \autoref{s.patterns}) we show that
$\Bnd_{28}(\oorb2)=\varnothing$.
(Note that $\bset(\orb)=\{0,\ldots,8,10,16\}$ for
$\orb\subset\oorb2$.)

\def\porb#1{\obj2{\cluster_{#1}}}

\subsection{Configuration~\hlink{24A1}{11}}\label{s.24A1-11}
We have $\ls|\stab\hh|=2160$ and
$\hh=\vv\CR$, where $\ls|\CR|=6$ and $\CR$ is \emph{not} a subset of an octad
(\cf. \autoref{s.24A1-2}).
This configuration is equivalent to \config{6D4}1 in $\N(6\bD_4)$,
see \autoref{s.6D4-1}, where we treat pseudo-geometric sets as well.

\def\porb#1{\obj5{\cluster_{#1}}}

\subsection{Configuration~\hlink{24A1}{12}}\label{s.24A1-12}
We have $\ls|\stab\hh|=192$ and $\hh=\cw\CO+\vv\CR$, where
$\CO$ is an octad and $\CR\subset\Omega\sminus\CO$ a
$4$-element set, so that there is \emph{no} octad $o\supset\CR$ such that
$\ls|o\cap\CO|=4$ (\cf. \autoref{s.24A1-5}).
This configuration is equivalent to \config{8A3}2 in $\N(8\bA_3)$,
see \autoref{s.8A3-2}, where we treat pseudo-geometric sets as well.
\endproof

\section{The Leech lattice}\label{Leech}

Let $N := \Lambda$ be the Leech lattice.
We prove the following theorem.

\theorem\label{th:Leech}
Let $\hbar \in \Lambda$, $\hbar^2 = 12$, and let $\fL \subset \fF(\hbar)$ be
a geometric set $\fL \subset \fF(\hbar)$. Then, unless
\roster*
\item $\ls|\fL| = 405$ and $\fL = \Lmax1$, see \eqref{eq.Lmax.1}, or
\item $\ls|\fL| = 357$ and $\fL = \Lsub2$, see \eqref{eq.Lsub.2}, or
\item $\ls|\fL| = 351$ and $\fL = \Lmisc1$, see \eqref{eq.Lmisc.1}
\endroster
one has $\ls|\fL|\le297$.
\endtheorem

\proof
It is well known (see, \eg, Theorem 28 in \cite[Chapter~10]{Conway.Sloane})
that any nonzero class $[\hbar]\in\Lambda/2\Lambda$ is represented
by a unique pair $\pm a\in\Lambda$,
where $a^2=4,6$ or $a^2=8$ and $a$ is part of a fixed coordinate frame.
Since, clearly, $a^2=\hbar^2\bmod4$, and since $\Lambda$ is positive definite
and root free, for $\hbar^2=12$ we have either
\roster
\item\label{a=4}
$\hbar = a + 2b$, where $a^2 = b^2 = 4$ and $a \cdot b = -2$
(\emph{type $6_{22}$} in \loccit.), or
\item\label{a=8}
$\hbar = a + 2b$, where $a^2= 8$, $b^2 = 4$, and $a \cdot b = -3$
(\emph{type $6_{32}$} in \loccit.)
\endroster
Besides, a pair $a,b\in\Lambda$ as in item~\iref{a=4} or \iref{a=8}
is unique up to $\OG(\Lambda)$.
Thus, there are two $\OG(\Lambda)$-orbits of square~$12$ vectors
$\hbar\in\Lambda$ (see Theorem~29 in \loccit.)

The two cases are considered below. We use the shortcut
$F:=\spn_2\fF(\hbar)$.

\def\hh{{\hyperref[a=4]{\hbar}}}

\subsection{Configuration~\pdfstr{1}{$\hyperref[a=4]1$}:
 \pdfstr{h = a mod 2L}{$\hh=a\bmod2\Lambda$}, $a\sp2=4$}\label{s.Leech-1}
We have
\[*
\ls|\OG_\hh(F)|=55180984320,\qquad
\ls|\fF|=891,\qquad
\rank F=23,\qquad
\discr F=\Cal{U}\oplus\bigl\<\tfrac74\bigr\>,
\]
and all index~$8$ extensions of $F\oplus\Z a$, $a^2=4$, are isomorphic and
have vector~$\hh$ of the same
type.
Listing iterated maximal subsets (see \autoref{s.maximal}) in the full set
$\fF(\hh)$, which is obviously complete and admissible,
after four steps we obtain $17$ complete sets
$\fL\subset\fF(\hh)$ of rank $\rank\fL\le21$ and size $\ls|\fL|\ge285$.
Using Nikulin's theory~\cite{Nikulin:forms}, one can easily obtain the
following statement.

\lemma\label{lem.mild.ext}
None of the
$17$ sets above has a proper root-free finite index extension $S\supset\spn_2\fL$
\rom(even abstract, not necessarily lying in~$\Lambda$\rom)
such that $\hh\in4S\dual$.
\done
\endlemma

Only
eight of the $17$ sets are geometric, namely,
the sets $\Lmax1$, $\Lsub2$, and $\Lmisc1$
(the subscript indicating the number of planes) described below, two sets of
size $297$ (ranks $20$ and $21$), and three sets of size $285$ (ranks $20$,
$21$, and~$21$).

\remark
Technically, we work
in~$F$ rather than in $\Lambda$ itself, which gives us a slightly larger symmetry
group: indeed,
$\OG_\hh(\Lambda)$ induces on~$F$ an index~$6$ subgroup of
$\OG_\hh(F)=\Aut\fF(\hh)$
(computed by the \texttt{GRAPE} package \cite{GRAPE:nauty,GRAPE:paper,GRAPE}
in \GAP~\cite{GAP4}).
That is why we have to consider \emph{abstract} finite index extensions.
\endremark

The three large sets found can be described
in terms of square~$4$ vectors in~$\Lambda$, which are relatively easy to
handle.
Consider the lattice $U:=\Z a+\Z b+\Z c$ with the Gram matrix
\[*
\bmatrix
  4&-2& 1\\
 -2& 4&-2\\
  1&-2& 4\endbmatrix,
\]
and let $V:=U+\Z v$, $v^2=4$, be its extension such that $v\cdot a=1$ and the
other two products are as follows:
\begin{alignat}5
&  v\cdot b=1,&&\quad v\cdot c=-2
&&\quad\mbox{for $\Lmax1$}:\quad&\ls|{\Aut\fL}|&=349920,\quad&T&\cong-[6,3,6],
\maxlabel1\\
&  v\cdot b=1,&&\quad v\cdot c=0
&&\quad\mbox{for $\Lsub2$}:\quad&\ls|{\Aut\fL}|&=10080,\quad&T&\cong- [2,1,18],
\sublabel2\\
&  v\cdot b=-2,&&\quad v\cdot c=1
&&\quad\mbox{for $\Lmisc1$}:\quad&\ls|{\Aut\fL}|&=31104,\quad&T&\cong-[6,0,6].
\misclabel1{351}
\end{alignat}
(For
the reader's convenience, we also list the size of the group $\Aut\fL$
of the graph automorphisms of~$\fL$,
computed \via\ \texttt{GRAPE}, and the \emph{transcendental lattice}
$T$, see \autoref{S.proofs} below.)
Up to $\OG_\hh(\Lambda)$, there is a unique isometry $V\into\Lambda$ such
that
$a+2b\mapsto\hh$.
Then, the set in question is
\[*
\fF(\hh)\cap\bigl((\Z\hh\oplus V^\perp)\otimes\Q\bigr).
\]


\def\hh{{\hyperref[a=8]{\hbar}}}

\subsection{Configuration~\pdfstr{2}{$\hyperref[a=8]2$}:
 \pdfstr{h = a mod 2L}{$\hh=a\bmod2\Lambda$}, $a\sp2=8$}\label{s.Leech-2}
We have
\[*
\OG_\hh(\Lambda)=\OG_\hh(F)=M_{24},\qquad
\ls|\fF|=759,\qquad
\rank F=24,\qquad
\discr F=\bigl\<\tfrac14\bigr\>\oplus\bigl\<\tfrac74\bigr\>.
\]
This lattice has two index~$4$ extensions: one of them is~$\Lambda$, and the
other, $\N(24\bA_1)$. Therefore, this configuration is equivalent
to~\config{24A1}{1} in $\N(24\bA_1)$, see \autoref{s.24A1-1}, where we treat
pseudo-geometric sets as well.
\endproof

\section{Proofs of the main results}\label{S.proofs}

In this section, we complete the proof of
the principal results of the paper, \viz.
Theorems~\ref{th.main} and~\ref{th.main_real}.

\subsection{Proof of \autoref{th.main}}\label{proof.main}

Let $X \subset \Cp5$ be a smooth cubic $4$-fold such that $\ls|\Fn X| > 350$.
Applying the replanting procedure (see \autoref{s.Neimeier_replanting})
to the $3$-polarized lattice $M_X \ni h_X$,
we obtain
a $12$-polarized lattice $S \ni \hbar$ embedded to
a Niemeier lattice $N$ (see \autoref{prop.replanting})
and such that
\begin{itemize}
\item one has $\hbar \in 4S\dual$,
\item the torsion $\Tors(N/S)$ is a $2$-group,
\item $S$ satisfies all conditions of \autoref{Niemeier_back}.
\end{itemize}
By Propositions \ref{planeclasses} and \ref{prop.S},
there is a bijection between the set of planes in $X$ and
the set
\[*
\fL=\fF(\hbar)\cap S \subset \fF(\hbar)
= \bigl\{l \in N \bigm|\text{$l^2 = 4$, $l \cdot \hbar = 4$}\bigr\}
\]
of $\hbar$-planes in $S$.
Furthermore, $\fL$
is geometric, essentially, by the definition.
Thus, Theorems \ref{th:few_components}, \ref{th:12A_2},
\ref{th:24A_1}, and \ref{th:Leech} imply that
we are in one of the following situations:
\roster*
\item $N \simeq \N(24\bA_1)$
and $\fL$ coincides with $\Lsub1$ up to $\OG_{\hh}(\bN)$,
\item $N \simeq \Lambda$
and $\fL$ coincides with $\Lmax1$ up to $\OG_{\hh}(\bN)$,
\item $N \simeq \Lambda$
and $\fL$ coincides with $\Lsub2$ up to $\OG_{\hh}(\bN)$,
\item $N \simeq \Lambda$
and $\fL$ coincides with $\Lmisc1$ up to $\OG_{\hh}(\bN)$.
\endroster
The set $\Lsub2$ is graph isomorphic to $\Lsub1$; this is established
by the \texttt{GRAPE} package \cite{GRAPE:nauty,GRAPE:paper,GRAPE} in
\GAP~\cite{GAP4}.

\begin{proposition}\label{proposition_405_357_351}
Let $\fL$
be one of the sets $\Lmax1$, $\Lsub2$ or $\Lmisc1$.
Then, there exists a unique {\rm (}up to $PGL(\C, 6)${\rm )} smooth cubic $X \subset \Cp5$ such that
$\Fn X$ is graph isomorphic to $\fL$.
\end{proposition}

%
%

\proof
Let
$\fL = \Lmax1$. The lattice $S = \spn_2 \fL \subset \Lambda$
is of rank $21$ and
has no
proper mild extension, see \autoref{lem.mild.ext}.
According to \autoref{Niemeier_back}, the $3$-polarized lattice $M \ni h$
obtained from $S \ni \hbar$ by the
inverse replanting procedure (\cf. \autoref{s.Neimeier_replanting})
admits a primitive embedding to $\bL$ with even orthogonal complement.
Thus, the existence of a smooth cubic $X \subset \Cp5$ such that
$\Fn X$ is graph isomorphic to $\fL$ follows from the surjectivity of the period map
(see \autoref{polarized_lattices}).

According to the global Torelli theorem~\ref{construction_maps},
the projective equivalence classes of smooth cubics $X \subset \Cp5$ such that
$\Fn X$ is graph isomorphic to $\fL$ (equivalently,
the $3$-polarized lattice $M_X \ni h_X$
is isomorphic to $M \ni h$, \cf. \autoref{lem.mild.ext})
are in a natural bijection with the
$O^+(\bL)$-orbits
of primitive embeddings
$M \hookrightarrow \bL$ such that $M^\perp$ is even,
where $O^+(\bL)$ is the group
of
auto-isometries of $\bL$ preserving the positive sign structure.

The classification of embeddings can be obtained using Nikulin's
theory~\cite{Nikulin:forms}.
For any embedding $M \hookrightarrow \bL$ such that $M^\perp$ is even,
the genus of the transcendental lattice $T = M^\perp$ is determined by
the discriminant $\discr M \cong -\discr T$. In our case, this implies that
$T \cong -[ 6, 3, 6 ]$.
The isomorphism classes of embeddings under consideration are in a bijection
with
$$
O(M) \backslash \Aut(\discr M) / O^+(T).
$$
Using \texttt{GRAPE}, one can check that the natural homomorphism
$$\Aut\fL = O_h(M) \to\Aut(\discr T)$$
is surjective, and to complete the proof of the uniqueness
there remains to notice that $T$ admits an orientation reversing autoisometry.

If $\fL = \Lmisc1$, the proof is
literally the same except that now $T \cong -[ 6, 0, 6 ]$.

Likewise,
in the case $\fL = \Lsub2$ the proof is literally the same except that now
we have $T \cong -[ 2, 1, 18 ]$. In view of \autoref{lem.mild.ext}, we do not
need to consider the graph isomorphic set $\Lsub1\subset\bN(24\bA_1)$, as it
results in the same lattice~$M$.
\endproof

As an immediate consequence of the uniqueness given
by \autoref{proposition_405_357_351},
we conclude that $\Lmax1$ is graph isomorphic the configuration of planes
in the Fermat cubic, see \autoref{section:Fermat}.

\lemma\label{lemma.HS}
The configuration of planes in the Clebsch--Segre cubic $Y$, see \autoref{section:HS},
is graph isomorphic to $\Lsub1 \simeq \Lsub2$.
\endlemma

\proof
The only other alternative would be $\Fn Y \simeq \Lmax1$.
However, $\Lmax1$ does not admit a faithful action of $\SG7$
(as $7!$ does not divide $\ls|\Aut \Lmax1| = 349920$). Alternatively,
by \autoref{real_structure} and \eqref{eq.Lmax.1}, the Fermat cubic does not admit a real structure
with respect to which the classes of the real planes span a lattice of rank $21$
(\cf. a more detailed argument in \autoref{proof.main_real}).
\endproof

\autoref{th.main} is an immediate consequence of \autoref{proposition_405_357_351}
and \autoref{lemma.HS}.
\qed


\subsection{Proof of \autoref{th.main_real}}\label{proof.main_real}

Let $Z \subset \Cp5$ be a smooth real cubic and $c\: Z \to Z$ the real
structure.
By \autoref{planeclasses_real}, the classes of the real planes in $Z$ span over $\Q$
the sublattice $M^c_Z := M_Z \cap \Ker(1-c^*)$.
Perturbing, if necessary, the period of $Z$ (see \autoref{polarized_lattices}),
we can assume that
$M_Z = M^c_Z$
\ie,
\emph{all} planes contained in~$Z$ are real.
Then, if $Z$ contains at least $357$ such planes,
\autoref{th.main} implies that
$Z$ is projectively equivalent to either the Fermat cubic of the Clebsch--Segre cubic.
In the latter case,
$T_Z\cong-[ 6, 3, 6 ]$
and the assumption that all planes in $Z$ are real contradicts
\autoref{real_structure}.
\qed

\subsection{Proof of \autoref{nodal}}\label{proof.nodal}

Given a cubic fourfold $Y$ with a node $p$, projection from $p$ gives a
birational map $$\pi: Y\dashrightarrow \PP^4$$ which contracts the tangent
cone of $p$ to a K3 surface $S\subset \PP^4$. As one easily checks, if
$P\subset Y$ is a plane, then $\pi(P)$ is either a line (if $P$ contains $p$)
or $\pi(P)$ is a plane in $\PP^4$ intersecting $S$ in a (possibly reducible) conic.\mnote{\2Dg:
irreducible?}
Now, given a
plane $V\subset \PP^4$, we claim that there can be at most one plane
$P\subset Y$ so that $V=\pi(P)$. Indeed, note that $H=\pi^{-1}(V)$ is a
$\PP^3$ passing through $p$. The cubic $F$ defining $Y$ splits when
restricted to $H$ and only two cases can occur: either $Y\cap H=P\cup P'\cup
P''$ for two planes $P',P''$ (in which case both $P'$ and $P''$ must pass
through $p$); or $Y\cap H=P\cup Q$ where $Q$ is an irreducible quadric
surface. In either case, there is at most one plane in $H$ that projects to
$P$, so we get the claim.

In all, we see that the number of planes is bounded by the sum of the number
of lines on $S$ (which is {\4at most} $42$, see \cite{degt:lines})
and the number of conics on $S$ (which is
{\4at most} $285$, see \cite{degt:conics}).
{\4Thus, altogether,}
there can be at most $42+285=327$ planes in a nodal
cubic.
{\4However, a sextic $K3$-surface with $261$ or more conics has no
lines (see \cite{degt:conics}); hence, the above bound is reduced down to
$42+260=302$.}
\qed

\example\label{ex.291}
{\4The $1$-parameter family of sextic $K3$-surfaces with the maximal number
$42$ of lines has a member with $249$ conics (counting both irreducible and
reducible ones).
Since the construction in
\autoref{proof.nodal}
is
obviously invertible, this gives us a nodal cubic fourfold with $291$ planes.
In fact, conjecturally, $249$ is the maximal number of conics on a sextic
$K3$-surface \emph{containing at least one line}; this fact motivates our
conjecture that $291$ is the sharp upper bound on the number of planes in a
nodal cubic fourfold.}
\endexample

{
\let\.\DOTaccent
\def\cprime{$'$}
\bibliographystyle{amsplain}
\bibliography{cubics}
}

\end{document}